\documentclass[11pt, reqno]{amsart}
\usepackage{mathrsfs}
\usepackage{amsmath}
\usepackage{enumerate}
\usepackage{graphicx}
\usepackage{fouridx}
\usepackage{amssymb,amscd}
\usepackage{mathtools}
\usepackage{multirow}
\usepackage{hyperref}
\usepackage{bm}
\usepackage{multicol}
\usepackage{multirow}
\usepackage[all, knot]{xy}
\usepackage{graphicx}
\usepackage{xcolor}
\usepackage{bm}

\usepackage{amssymb}
\usepackage{tikz-cd} 
\usepackage{amsthm}
\usepackage{hyperref}
\usepackage{amsmath,amssymb,amsopn,amsfonts,mathrsfs,amsbsy,amscd}
\usepackage{longtable}
\usepackage{xcolor}
\usepackage{graphicx}

\usepackage{bm}
\usepackage{eucal}
\usepackage{amscd}

\usepackage[active]{srcltx} 

\newcommand {\emptycomment}[1]{} 

\newcommand{\nc}{\newcommand}
\newcommand{\delete}[1]{}
\nc{\mfootnote}[1]{\footnote{#1}} 
\nc{\todo}[1]{\tred{To do:} #1}

\nc{\tred}[1]{\textcolor{black}{#1}}
\nc{\tblue}[1]{\textcolor{blue}{#1}}
\nc{\tgreen}[1]{\textcolor{green}{#1}}
\nc{\tpurple}[1]{\textcolor{purple}{#1}}
\nc{\btred}[1]{\textcolor{black}{\bf #1}}
\nc{\btblue}[1]{\textcolor{blue}{\bf #1}}
\nc{\btgreen}[1]{\textcolor{green}{\bf #1}}
\nc{\btpurple}[1]{\textcolor{purple}{\bf #1}}

\nc{\na}[1]{\textcolor{blue}{Nacer:#1}}

\newtheorem{defn}{Definition}[section]
\newtheorem{thm}{Theorem}[section]

\newtheorem{prop}{Proposition}[section]
\newtheorem{lem}{Lemma}[section]

\newtheorem{rem}{Remark}[section]
\numberwithin{equation}{section}

\newtheorem{Cases}{\textbf{Case}}
\newtheorem{scass}{\textbf{Subcase}}[Cases]
\font\bb=msbm10

\def\R{\hbox{\bb R}}

\begin{document}
	\title[	Second cohomology group and (quadratic) extensions of 	
	(metric) Hom-Jacobi-Jordan algebras]{	Second cohomology group and  (quadratic) extensions of 	
	(metric)	 Hom-Jacobi-Jordan algebras}
	
	\author[N. Saadaoui]{Nejib Saadaoui}
	\address{Institut Supérieur d'Informatique et de Multimédia de  Gabès,\\ \footnotesize
		Campus universitaire-BP 122, 6033 Cité El Amel 4, Gabès, TUNISIE.}
	\email{najibsaadaoui@yahoo.fr}
	
	
	
	\maketitle

	
	
	\tableofcontents
	
	\numberwithin{equation}{section}
	
	\tableofcontents
	\numberwithin{equation}{section}
	\allowdisplaybreaks

	\begin{abstract}
The main purpose of this paper is to provide a second group cohomology of a (metric) Hom-Jacobi-Jordan
algebra with coefficients in a given representation. 
We show that this second cohomology group
classifies abelian extensions of a Hom-Jacobi-Jordan algebra by a representation. As an application, we classify the low dimensional multiplicative Hom-Jacobi-Jordan
algebras.


		
		


	\end{abstract}

	\section*{Introduction}
	The concept of a Hom-Lie algebra was studied first, from mathematical viewpoint
	by Hartwig,
	Larsson, and Silvestrov in \cite{HLS}, which is a non-associative algebra satisfying
	the skew symmetry and the $\sigma$-twisted Jacobi identity. When $\sigma= id$, the
	Hom-Lie algebras degenerate to exactly the Lie algebras. In this sense, it
	is natural to seek for possible generalizations of known theories from Lie to
	Hom-Lie algebras: Hom-Lie algebras structures was defined in \cite{Ideal1,structure}.  As well as, derivations, representations
	and cohomology of Hom-Lie algebras was studied in \cite{shengR}. Cohomology and deformations
	of Hom-Lie algebras was studied  in \cite{MakloufSil,AmmarMakloufZey}.
	The Jacobi-Jordan-algebra is introduced by Dietrich Burde and  Alice Fialowski 	in\cite{Dietrich.Alice} which is
	a non-associative algebra satisfying the symmetric condition and the Jacobi identity algebras.  The set of these algebras is a subclass of a class of Jordan-Lie superalgebras introduced
	in \cite{Kamia}. Mokh Lie algebras is another name for these algebras (\cite{Pasha}).
	
	The difference with Lie algebras is thus shown in the symmetrization and therefore
	view to the importance of Hom algebras in several domains it is important to generalize the results determined in Jacobi-Jordan algebras
	to Hom -Jacobi- Jordan algebras.
	Cyrille started this generalization in \cite{Cyrille}, in this paper we are interested to studied the second group of cohomology of these algebras
	and its relation with extensions.
	The paper is organized as follows. In the first section, we recall the definition of Hom-Jacobi-Jordan algebras and we give 
	the classification of Hom-Jacobi-Jordan algebras of dimension $2$. In section $2$, we search
	the conditions for the direct sum of two vector spaces M and V to be a Hom-Jacobi-Algebra. 
	During, this investigation, we discover the definitions of representation of $2$-cocycle. looking about conditions to have 
	two direct sums $(M,d,\gamma)$ and $(M,d',\gamma)$ of Hom-Jacobi-Algebra are isomorphic, we find the definition of the operator 
	of $1$-cobords as well as the space of $2$-cobords.
	In order to classify these Hom Jacobi algebras of type direct sum, we need to add certain conditions on this ismorphism, 
	which brings us to define the extensions in the following section 3.
	Moreover in section 3, we generalize a known result of Lie case, that is the class of the extensions is in bijection with the second cohomology group.
	In the last section, we are interested in a particular example of extensions, called central extension and we give the central extensions of 2
	dimensional Jacobi-Algebras.			

\section{Hom-Jacobi-Jordan algebras  }
In this section, first we recall some basic facts about Hom-Jacobi-Jordan
algebras. 
\begin{defn}(\cite{Cyrille})
	A Hom-Jacobi-Jordan algebra  is a triple $(J,[\cdot,\cdot] , \alpha)$, where $ J $ is  a vector space eqquiped with a symmetric    bilinear map  $ [\cdot,\cdot]\colon J\times J\rightarrow J$  and a linear map  $ \alpha\colon J\rightarrow J \ $ such that
	\begin{equation}
		\left[  \alpha(x),[y,z]\right] 
		+\left[ \alpha(y),[z,x] \right]  
		+\left[  \alpha(z),[x,y]	\right]  	
		=0\label{jacobie}
	\end{equation}
	for all   $x, y, z$ in $J$, which is called Hom-Jacobi identity.\\
	We recover Jacobi-Jordan algebras when the linear map $ \alpha $ is the the identity map.\\
	A  Hom-Jacobi-Jordan-algebra   is called abelian if the bilinear map $ [\cdot,\cdot] $ vanishes for all elements in $ J $.\\
	A Hom-Jacobi-Jordan algebra is called a multiplicative Hom-Jacobi-Jordan if $ \alpha $ 
	is an algebraic morphism with 
	\begin{equation}
		\alpha\left( [x,y]\right)= \left[ \alpha(x),\alpha(y)\right]  \label{multiplicative}	                        
	\end{equation}	
	for any $ x,\ y\in J $.\\
A Hom-Jacobi-Jordan algebra is called a regular  if $\alpha$ is an algebra
automorphism.\\	

Let $V$ be a vector space. A $k$-linear map  $f\colon\underbrace{J\times J\ldots\times  J}_{k\ \text{times}}\to V,$ is said to be
	symmetric  if:
	$${\displaystyle \;f(x_{\sigma (1)},\cdots ,x_{\sigma (k)})=f(x_{1},\cdots ,x_{k})}\mbox{ for all } \sigma \in  \mathfrak{S}_{k},$$
	$\mbox{ where } {\mathfrak S}_{k} \mbox{ is the group of  permutations of }\{1,\cdots,k\}$.  For $k\in \mathbb N$, the set of symmetric $k$-linear maps is denoted by $S^k(  J,V)$. 	
\end{defn}

\begin{defn}
A homomorphism of Hom-Jacobi-Jordan algebras $\phi \colon \left( J,[\cdot,\cdot],\alpha\right)\to  \left( J',[\cdot,\cdot]',\alpha'\right) $	
is a linear map $\phi \colon J\to J'  $ such that 
\begin{align}
\phi \circ \alpha =\alpha' \circ \phi\\
	\phi\left( [x,y]\right)=\left[ \phi(x),\phi(y) \right]'	\label{isom}                         
\end{align}
for all $x,y\in J$.	The Hom-Jacobi-Jordan algebras $\left( J,[\cdot,\cdot],\alpha\right)  $
and $\left( J',[\cdot,\cdot]',\alpha'\right) $
are isomorphic if there is
a Hom-Jacobi-Jordan algebras homomorphism $\phi \colon \left( J,[\cdot,\cdot],\alpha\right)\to  \left( J',[\cdot,\cdot]',\alpha'\right) $
such that $\phi \colon J\to J'  $ is bijective.		                      
\end{defn}
\begin{defn}
Let $ \left(J,[\cdot,\cdot] ,\alpha\right)  $ be a Hom-Jacobi-Jordan algebra.
A subspace $I$ of $J$ is said to be an ideal if for $x\in I$ and $ y\in J $
we have $ [x,y] \in I$ and $ \alpha(x) \in I$, Moreover, $I$ is called a  abelian ideal of $ J $ if $ [I,I]=0 $.
A Hom-Jacobi-Jordan  subalgebra $H$ of $J$ is a linear 	subspace such that $ [H,H]\subset H  $ and $ \alpha(H)\subset H.$
\end{defn}
\begin{defn}
Let $\left(J,[\cdot,\cdot],\alpha \right) $ be a multiplicative  Hom-Jacobi-Jordan algebra. We put $ D^0(J) =J$, $ D^1(J) =[J,J]$ and, more generally $ D^{k+1}(J)=[D^k(J),D^k(J)] $, for every $k\geq 0$.  
All these subspaces are ideals
of $ J $ and we have the following decreasing sequence, called the derived sequence
\[ J=D^0(J)\supset D^1(J)\supset\cdots \supset D^k(J) \supset  \cdots \]
A Hom-Jacobi-Jordan algebra $J$ is called solvable 	if there is a integer $k$ such that $D^k(J)=\{0\}$.
\end{defn}	
\begin{prop}
Any solvable Hom-Jacobi-Jordan algebra	has an abelian ideal.	               
\end{prop}
 \begin{proof}            
For $j=k$,  we have $ [D^k(J),D^k(J)]\subset D^{k+1}(J) $ and $D^{k+1}(J)=\{0\} $.
\end{proof}
\begin{defn}
The center of a Hom-Jacobi-Jordan algebra $\left( J,[\cdot,\cdot],\alpha\right) $ is the  subspace \[\mathfrak Z(J)=\{x\in J\mid [x,y]=0,\forall y\in J  \}.\] 
\end{defn}

\section{The second  cohomology group of  Hom-Jacobi-Jordan algebras}\label{2group}
The first and second cohomology operators of Jacobi-Jordan algebra was introduced in \cite{Pasha}. In the following,  we 
give a construction of first and second cohomology operators of  Hom-Jacobi-Jordan algebra. Then, we define their second  cohomology group.\\

Let $\left(M,[\cdot,\cdot]_{M},\alpha_{M} \right) $  be a finite dimensional solvable multiplicative Hom-Jacobi-Jordan algebra, and let $V$ be an abelian ideal of $M$. Assume that there exists a subspace   $J$ of $ M $ such that 
 $J\oplus V=M$ and $\alpha_{M}(J)\subset J  $. Therefore, we can define the linear map $ \alpha\colon J\to J $ by $ \alpha(x)=\alpha_{M}(x) $. Define a
symmetric bilinear maps $[\cdot,\cdot]\colon J\to J$ by $[x,y]=\pi_{J}\left( [x,y]_{M} \right)  $ where $  \pi_{J}$ is the natural projection of $M$ onto $J$.
 Then, it is clear that $\left(J,[\cdot,\cdot],\alpha \right) $ is a Hom-Jacobi-Jordan algebra.
 \subsection{Construction of the $2$-coboundary operator}
  Since $V$ is an ideal of $M$, we can define the linear maps $\rho\colon J\to End(V)$ by $\rho(x)v=[x,v]_{M}$ and $ \beta\colon V\to V $ by $\beta(v)=\alpha_{M}(v)$.  Then, the triple $\left(V,\rho,\beta \right) $ is called representation of $J$ on $V$. Since $ \left(M,[\cdot,\cdot]_{M},\alpha_{M} \right) $ is multiplicative,
$\alpha_{M}\left( [x,v]_{M}\right)=\left[\alpha_{M}(x),\alpha_{M}(v) \right]_{M}  $. This implies 
$ \beta\left( \rho(x)v\right) =\rho\left( \alpha(x)\right) \beta(v) $. Moreover, by the 
Hom–Jacobi identity,  for all  $x,y\in J  $ and  $ v\in V $, we have
$	\rho\left([x,y] \right)\circ \beta (v) +\rho\left(\alpha(x)\right)\rho(y)(v)+  \rho\left( \alpha(y)	\right) \circ \rho(x)v=0.  $ Therefore, the definition of representation may be overwritten as follows.
\begin{defn}\label{prilimenaryRep}
A representation of a Hom-Jacobi-Jordan algebra $(J,[\cdot,\cdot],\alpha)$ on a vector space $V$
with respect to $ \beta\in End(V)   $ is a linear map 
$\rho\colon J\to End(V)$, such that 
for any $x\in J$, $ v\in V $,  the following equalities are satisfied: 
\begin{align}
	\rho(\alpha(x))\circ \beta&=\beta\circ \rho(x)\label{rep1}\\
	\rho\left([x,y] \right)\circ \beta &=-\rho\left(\alpha(x)\right)\rho(y)-  \rho\left( \alpha(y)	\right) \circ \rho(x). \label{rep2}                      
\end{align}
 We denote a representation of a Hom-Jacobi-Jordan algebra $(J,[\cdot,\cdot],\alpha)$ by a triple $ \left(V,\rho,\beta \right)  $.	            
\end{defn}
Our next goal is to define the set of $2$-Hom-cochain of $J$ on $V$.
Define a
symmetric bilinear maps $\theta\colon J\to J$ by $\theta(x,y)=\pi_{V}\left( [x,y]_{M} \right)  $ where $  \pi_{V}$ is the natural projection of $M$ onto $V$.
Since the Hom-Jacobi-Jordan algebra $\left(M,[\cdot,\cdot]_{M},\alpha_{M} \right) $ is multiplicative, $[\cdot,\cdot]_{M}$ and $\alpha_{M}  $ satisfies 
Eq.\eqref{multiplicative}, it follows that
$ \alpha_{M}\left([x,y]_{M} \right)=\left[\alpha_{M}(x),\alpha_{M}(y) \right]_{M}   $ for all $x,y\in J$. Then, $ \alpha\left([x,y] \right)=\left[\alpha(x),\alpha(y) \right]   $ and $\beta\left(\theta(x,y) \right)=\theta\left( \alpha(x),\alpha(y)\right) $. The last equality	is called $2$-cochain condition and it denoted  $ \beta\circ \theta =\theta\circ\alpha $. 
\begin{defn}
	A symmetric bilinear map 
	with values in $V$ is defined to be a $2$-cochain $f\in S^{2}(J, V)$ such it is compatible with $ \alpha $ and $ \beta $ in the sense that $ \beta\circ f =f\circ\alpha $.  Denote $C^{2}_{\alpha,\beta}(J,\ V)  $ the set of $2$-cochain:
	\begin{equation*}
		C^{2}_{\alpha,\beta}(J,\ V)=\left\lbrace f\in S^2(J,V)\mid \beta\circ f =f \circ\alpha \right\rbrace  .\\                             
	\end{equation*}
\end{defn}
Finally, we define the $2$ cocycles on $V$ followed by the $2$ cobord.
We have 
\begin{equation*}
\left[\alpha_{M}(x) ,[ y,z ]_{M} \right]_{M}
=[\alpha(x),[y,z]]+\theta\left(\alpha(x),[y,z] \right) +\rho(\alpha(x))\theta(y,z)      
\end{equation*}
according to the Hom Jacobi identity  for $ x,y,z\in J $,

\begin{gather*}
\underbrace{ [\alpha(x),[y,z]]+[\alpha(y),[x,z]]+[\alpha(z),[y,z]]}_{\in J}\\
+\underbrace{\theta\left(\alpha(x),[y,z] \right)
+\theta\left(\alpha(y),[x,z] \right)
+\theta\left(\alpha(z),[x,y] \right)}_{\in V}\\
+\underbrace{\rho(\alpha(x))\theta(y,z)+\rho(\alpha(y))\theta(x,z) 
+\rho(\alpha(z))\theta(x,y)	}_{\in V}=  0              
\end{gather*}
Hence, 
\begin{equation}\label{cocycle}
	\begin{array}{ll}
\theta\left(\alpha(x),[y,z] \right)
+\theta\left(\alpha(y),[x,z] \right)
+\theta\left(\alpha(z),[x,y] \right)\\
+\rho(\alpha(x))\theta(y,z)+\rho(\alpha(y))\theta(x,z) 
+\rho(\alpha(z))\theta(x,y)=0
	\end{array}
\end{equation}
Therefore, we say that the $
2$-cochain satisfies $2$-cocycle condition and the next definitions are provided.
\begin{defn}\label{pr2cocycle}
	We call $2$-coboundary operator of Hom-Jacobi-Jordan algebra $ J$ 
	the map \[d^2: C^{2}_{\alpha,\beta}(J,V)\rightarrow S^{3}(J,V),\qquad f\mapsto d^2(f) \]
	defined by
	\begin{align}\label{1adjoint2}
		d^{2}(f)(x,y,z)&=f\left(\alpha(x),[y,z] \right)
		+f\left(\alpha(y),[x,z] \right)
		+f\left(\alpha(z),[x,y] \right)\\
		+&\rho(\alpha(x))f(y,z)+\rho(\alpha(y))f(x,z) 
		+\rho(\alpha(z))f(x,y).\nonumber
	\end{align}
A $2$-Hom-cochain $ f $ is a $2$-cocycle	 if and only if $ d^{2}(f)=0 $. Hence, $ \ker(d^2)=Z_{\alpha,\beta}^2(J,V) $ where $Z_{\alpha,\beta}^2(J,V)  $ is the set of $2$-cocycle of the Hom-Jacobi-Jordan algebra with coefficients in the representation $ (V,\rho,\beta) $.
\end{defn}
\subsection{Construction of the $1$-coboundary operator}
Let $[\cdot,\cdot]'_{M}\colon M\times M\to M  $ be a symmetric bilinear map 
such that 
$\left(M,[\cdot,\cdot]'_{M},\alpha_{M} \right) $ is a Hom-Jacobi-Jordan  algebra. Then, the linear map $ \theta'=\pi_{V}\circ [\cdot,\cdot]'_{M} $ is a $2$-cocycle of $J$ on $V$.
Suppose there exists an isomorphism  of Hom-Jacobi-Jordan
algebra  $\Phi \colon \left( M,[\cdot,\cdot]_{M},\gamma\right) \to \left( M,[\cdot,\cdot]'_{M},\gamma\right) $.
 We define the linear maps $s\colon J\to J  $ by 
$s(x)=\pi_{J}\circ \Phi(x) $, $i\colon J\to V  $  by 
$i(x)=\pi_{V}\circ \Phi(x) $,
 $s'\colon V\to J  $  by 
$s'(v)=\pi_{J}\circ \Phi(v) $,
 and $i'\colon V\to V  $ by 
$i'(v)=\pi_{V}\circ \Phi(v) $. 
 We have  $ \Phi\left( [x,y]_{M}\right)=\left[\Phi(x),\Phi(y) \right]'_{M}   $.
 Then, $ \Phi\left( [x,y]+\theta(x,y)\right)=\left[s(x)+i(x),s(y)+i(y) \right]'_{M}   $. Therefore,
 \begin{gather*}
\underbrace{ s\left([x,y] \right)+s'\left( \theta(x,y)\right)}_{\in J}
 +\underbrace{ i\left( [x,y]\right) +i'\left( \theta(x,y)\right)}_{\in V} 
 =\underbrace{\left[s(x),s(y) \right]}_{\in J}\\
 +
 \underbrace{\rho\left(s(x) \right)i(y)+\rho\left(s(y) \right)i(x) 
 + \theta'\left(s(x),s(y) \right) }_{\in V}.  	                    
 \end{gather*}
This gives
\begin{gather*}
s\left([x,y] \right)=\left[s(x),s(y) \right]-s'\left( \theta(x,y)\right)	               
\end{gather*}
and
\begin{gather}\label{f1}
i'\left( \theta(x,y)\right)-\theta'\left(s(x),s(y) \right)
=-i\left( [x,y]\right)+\rho\left(s(x) \right)i(y)+\rho\left(s(y) \right)i(x).	
\end{gather}
We assume that $s(J)\cap s'(V)=\{0\}  $. So $s(J) \oplus s'(V)=J  $ and we can define a linear map $f\colon J\to V$ by  
\begin{gather}\label{f}
f(s(x))=i(x) \text{  and  }     f(s'(v))=i'(v).	                                      
\end{gather}
Moreover, assume that $ i'\left( \theta(x,y)\right) \in s'(V)$.
Therefore, we obtain
	\begin{gather*}
f\left( s'(\theta(x,y))\right)-\theta'\left(s(x),s(y) \right)
=-f\left(s ([x,y])\right)+\rho\left(s(x) \right)f(s(y))+\rho\left(s(y) \right)f(s(x))	                
	\end{gather*} 	
Hence,
	\begin{align*}
	\theta'\left(s(x),s(y) \right)
	&=f\left( s'(\theta(x,y))\right)+f\left(s ([x,y])\right)-\rho\left(s(x) \right)f(s(y))-\rho\left(s(y) \right)f(s(x))\\
	&=f\left( \left[s(x),s(y) \right]\right)-\rho\left(s(x) \right)f(s(y))-\rho\left(s(y) \right)f(s(x)). 	                
\end{align*}
Thus,
	\begin{gather*}
	\theta'\left(x',y' \right)
	=f\left( \left[x',y' \right]\right)-\rho(x')f(y'))-\rho(y' )f(x')	                
\end{gather*}
for all $x',y'\in s(J).$\\
Now, let’s determine the properties of $f$.
Since $\Phi$ is an isomorphism of Hom-Jacobi-Jordan algebra, we have $ \Phi\alpha_{M}(x)=\alpha_{M}(\Phi(x)) $.
Hence, 
	$s(\alpha(x))=\alpha(s(x))$	                            
and 
	$i\left(\alpha(x) \right) = \beta(i(x))$.	
Thus, $ f\left(\alpha( s(x))\right)=\beta f\left( s(x)\right) $ for all $x\in J$.
Therefore, the definitions of $1$-Hom cochain and $1$-cobord are as follows:
\begin{defn}
We say that $f$ is a $1$-cochain if $ f\colon J\to V $
 is
a linear  map satisfying $ f\circ \alpha=\beta \circ f $.
The  space of all $1$-cochains will be denoted by $ C_{\alpha,\beta}^{1}(J,V) $.\\
We call $1$-coboundary operator of Hom-Jacobi-Jordan algebra $ J $ 
the map \[ d^1: C^{1}_{\alpha,\beta}(J,V)\rightarrow S^{2}(J,V),\qquad f\mapsto d^1f \]
defined by 
\begin{eqnarray}\label{adjoint1}
	d^{1}(f)(x,y)&=&f([x,y])-\rho(x)f(y)-\rho(y)f(x).
\end{eqnarray}	                  
\end{defn}
\subsection{Second cohomology group}
Let 
$ \left(V,\rho,\beta \right)  $
 be a representation of a Hom-Jacobi-Jordan algebra $ \left(J,[\cdot,\cdot],\alpha \right)  $. 
\begin{thm}\label{cobord} With the above  notation  in this section, we have
	$ d^{2}\circ d^1=0 $ .                  
\end{thm}
\begin{proof}
	We Have
	\begin{eqnarray*}
		d^{1}(f)(x,y)&=&f\left( [x,y]\right) -
		\rho(x)f(y)	-\rho(y)f(x).
	\end{eqnarray*}
	Then
		\begin{align}
			&	d^2\circ d^1(f)(x,y,z)\nonumber\\	
			&=	d^1(f)([x,y],\alpha(z))
			+d^1(f)\left( [x,z],\alpha (y)\right) 
			+d^1(f)\left( \alpha (x),[y,z]\right) \nonumber\\	
			+&\rho\left( \alpha(x)\right) d^1(f)(y,z)
			+\rho\left( \alpha(y)\right) d^1(f)(x,z)
			+\rho\left( \alpha(z)\right) d^1(f)(x,y)\nonumber\\	
			&=f\left( \left( [x,y],\alpha(z)\right) \right)
			+f\left([x,z],\alpha (y)\right)
			+f\left([y,z],\alpha (x)\right)
			\label{cobordJacobi}\\
			&-\rho\left([x,y] \right) f(\alpha(z))
			-\rho\left(\alpha(x) \right)\rho(y) f(z)
			-\rho\left(\alpha(y) \right)\rho(x) f(z)
			\label{cobordCochain1}\\ 
			&-\rho\left([x,z] \right) f(\alpha(y))
			-\rho\left(\alpha(x) \right)\rho(z) f(y)
			-\rho\left(\alpha(z) \right)\rho(x) f(y)
			\label{cobordCochain2}\\
			&-\rho\left([y,z]\right) f(\alpha(x))
			-\rho\left(\alpha(z) \right)\rho(y) f(x)
			-\rho\left(\alpha(y) \right)\rho(z) f(x)
			\label{cobordCochain3}\\
			&-\rho\left( \alpha(z)\right) f\left( [x,y]\right) 
			+\rho\left( \alpha(z)\right) f\left( [x,y]\right) 
			-\rho\left( \alpha(y)\right) f\left( [x,z]\right)\\
			&+\rho\left( \alpha(y)\right) f\left( [x,z]\right) 
			-\rho\left( \alpha(x)\right) f\left( [y,z]\right)
			+\rho\left( \alpha(x)\right) f\left( [y,z]\right) 
		\end{align}	 		
		By the Hom-Jacobi identity, we obtain that Eq.\eqref{cobordJacobi}=0. \\
		Since $ f\circ \alpha=\beta \circ f $  and the linear map $ \rho $  satisfies  \eqref{rep2}, we have that  Eq.\eqref{cobordCochain1}=0,  Eq.\eqref{cobordCochain2}=0  and Eq.\eqref{cobordCochain3}=0. The sum of the other six items is zero obviously. Therefore, we have $  d^2\circ d^1(f)(x,y,z)=0$. The
		proof is completed.         
	\end{proof}
\begin{defn}
The $2^{th}$ cohomology  group is the quotient $  H^{2}_{\alpha,\beta}(J,V)=Z^{2}_{\alpha,\beta}(J,V)/ B_{\alpha,\beta}^{2}(J,V)$ where $Z_{\alpha,\beta}^{2}(J,V)=\ker \ d^{2}$ and $B^{2}_{\alpha,\beta}(J,V)=Im \ d^{1}.$
\end{defn}
\section{Representation of Hom-Jacobi-Jordan algebras}
In this section,  we give some examples of representations 
that we will need in the remainder of the paper.

\subsection{Representation on V'=End(J,V)}	
Let $V'=End(J,V)$ be the  vector space of linear map  $f\colon   J\to  V $.
We  define the linear map 
$ \beta'\colon V'\to V'  $ by
\begin{equation}\label{alphaprime}
\beta'(Z)=Z\left( \alpha(\cdot)\right)	      
\end{equation}
and
the  linear map $\rho'\colon J\to End(V')$   by 
\[\rho'(x)Z=Z\left( [x,\cdot]\right) .\]
Let $Z\in V'$, $ x,y\in J $. We compute the right hand side of the identity 
\eqref{rep2}, we obtain
\begin{gather*}
-\rho'\left( \alpha(x)\right) \rho'(y)Z-\rho'\left( \alpha(y)\right) \rho'(x)Z
=-Z\big(\left[  y,[\alpha(x),        \cdot]  \right]             \big) 
-Z\big(\left[   x,[\alpha(y),        \cdot] \right]             \big). 	              
\end{gather*}
The left hand side of \eqref{rep2} gives us 
\begin{equation*}
\rho'\left( [x,y]\right) 	\beta'(Z) 
=Z\big(\alpha\big(\left[   [x,y],\cdot  \right]   \big)                   \big). 
\end{equation*}
Therefore,  we have the following observation.
\begin{prop}\label{prepcoadjoint}
The triple	$ \left( V',\rho',\alpha' \right)  $	is a representation of $J$ if and only if 
\begin{equation}\label{repcoadjoint}
	\alpha\big(\left[   [x,y],\cdot  \right]   \big)  =
	-\left[  y,[\alpha(x),        \cdot]  \right]            
	-\left[   x,[\alpha(y),        \cdot] \right]          		                                  
\end{equation}
for all $x,y,z\in J$. In this case, $ \left( V',\rho',\alpha' \right)  $ is called the  generalized coadjoint representation. 	                  
\end{prop}
Associate to the generalized coadjoint representation $ \rho' $, the coboundary operator $ d^{1}\colon C_{\alpha,\beta}^1\to C_{\alpha,\beta}^2$ (resp. $ d^{2}\colon C_{\alpha,\beta}^2\to C_{\alpha,\beta}^3 $)  defined in \eqref{adjoint1} (resp. \eqref{1adjoint2} ) are given by 
\begin{equation*}
d'^{1}\colon C_{\alpha,\alpha'}^1\to C_{\alpha,\beta'}^2;		d'^{1}(f)(x,y)=f([x,y])
-f(y)\big([x,\cdot] \big)
-f(x)\big([y,\cdot] \big)
\end{equation*}
and $d'^{2}\colon C_{\alpha,\alpha'}^2\to C_{\alpha,\alpha'}^3$;
\begin{align*}
d'^{2}g(x,y,z) 
&= g([x,y],\alpha(z))
+g([x,z],\alpha(y))
+g([y,z],\alpha(x))\\
+&g(x,y)([\alpha(z),\cdot])
+g(x,z)([\alpha(y),\cdot])
+g(y,z)([\alpha(x),\cdot]).
\end{align*}
Hence, by Theorem \ref{cobord} we have
\begin{equation}\label{cobordcoadjoint}
d'^{2}\circ d'^{1}=0 .               
\end{equation}
In the particular case when $V=\R$, we obtain the dual space $J^*$ and we denote 
\begin{align*}
C_{r}^2(J,\R)&=\{ f \text{ bilinear form } \mid f(x,\cdot) \in C_{\alpha,\alpha'}^1(J,J^*), \forall x\in J        \}\\ 	
C_{r}^3(J,\R)&=\{ f \text{ trilinear form } \mid f(x,y,\cdot) \in C_{\alpha,\alpha'}^2(J,J^*) , \forall x, y\in J         \}\\  	
C_{r}^4(J,\R)&=\{ f\text{  $4$-linear form } \mid f(x,y,z,\cdot) \in S^3(J,J^*) , \forall x, y,z\in J         \}   . 	              
\end{align*}		
Define $ d_{r}^2\colon C_{r}^2(J,\R)\to C_{r}^3(J,\R) $ by 
\begin{gather}\label{operatorCoadjoint4}
d_{r}^2f(x,y,t) 
=f([x,y],t)-f(y,[x,t])-f(x,[y,t]).
\end{gather}
and define $ d_{r}^3\colon C_{r}^3(J,\R)\to C_{r}^4(J,\R) $ by
\begin{align}\label{operatorCoadjoint3}
d_{r}^3\gamma(x,y,z,t)&=
\gamma(d(x,y),\alpha(z),t)
+\gamma([x,z],\alpha(y),t)
+\gamma([y,z],\alpha(x),t)\\	
+& \gamma\left( x,y,[\alpha(z),t] \right)
+\gamma\left( y,z,[\alpha(x),t] \right)
+\gamma\left( x,z,[\alpha(y),t] \right)\nonumber
\end{align}
\begin{thm}\label{CohomologyReal}
With the above notations, we have
$d_{r}^3\circ d_{r}^2=0$.
\end{thm}
\begin{proof}
We have, $	d_{r}^2f(x,y,t) 
=d'^{1}f(x,y)(t)  $ and $	d_{r}^3f(x,y,z,t) 
=d'^{2}f(x,y,z)(t)$. Then, by \eqref{cobordcoadjoint} we obtain $ d_{r}^3\circ d_{r}^2=0. $
\end{proof}	
\begin{prop}\cite{saadaoui}
Let $(V,\rho,\beta )$ be a representation of a Hom-Jacobi-Jordan algebra $(J,[\cdot,\cdot],\alpha)$ and  $\theta$ a  $2$-cocycle of $J$ on $V$.
Let $ \left(M,[\cdot,\cdot]_{M},\alpha_{M} \right)  $ be the extension of $J$ by $V$ by means of $\theta$.            
The triple $ \left(V',\rho',\beta' \right)  $, where $ \rho'\colon M\to End(V') $ is given by $ \rho'(x+v)f(\cdot))=f([x+v,\cdot]_{M}) $ and $ \beta'\colon V'\to V' $ is given by $ \beta'(f)=\beta \circ f   $,  defines a representation of the Hom-Jacobi-Jordan algebra $ \left(M,[\cdot,\cdot]_{M},\alpha_{M} \right)  $
if and only if 
\begin{align}
	\alpha\Big(\big[[x,y],t\big]       \Big)&=-	\Big[x, [\alpha(y),t]    \Big] 
	-\Big[y, [\alpha(x),t]    \Big]\label{coadjointprime1}\\
	\beta\Big(\rho(t)\theta(x,y)      \Big)&=
	-\rho\left( x\right) \theta (\alpha(y),t)
	-\rho\left( y\right) \theta (\alpha(x),t)\label{coadjointprime2}.                       
\end{align}    
\end{prop}
Define 
$d^1_c\colon  C^{1}_{\alpha,\beta}(J,V)\to S^{2}(J,V) $
by 		\begin{eqnarray*}
d^{1}(f)(x,y)&=&f\left( [x,y]\right) -
\rho(x)f(y)	-\rho(y)f(x).
\end{eqnarray*}
and 
$d^2_c\colon S^{2}(J,V)\rightarrow C^{3}(J,V)$
by
\begin{align}\label{adjoint2}
d_{c}^{2}(f)(x,y,z)&=	\left(x,[\alpha(y),z] \right)
+f\left(y,\delta(z,\alpha(x)) \right)
+f\Big(\theta\left(z,[x,y]\right)\Big) \\
+& 	\rho\left(x \right) f(\alpha(y),z)
+\rho\left(y \right) f(z,\alpha(y))
+\beta\Big(\rho\left(z \right) f(x,y) \Big).\nonumber
\end{align}
where $C^3(J,V)=\{\gamma\in Hom(J^3,V)\mid \gamma\left(x,y,t \right)=\gamma\left(y,x,t \right)\}$ .
\begin{thm}
We have $ 	d_{c}^{2}\circ d_{c}^1=0 $.    
\end{thm}
\begin{proof}
It is straightforward.
\end{proof}

\section{Extensions of Hom-Jacobi-Jordan algebras}
 Let 
$ \left(V,\rho,\beta \right)  $
be a representation of a Hom-Jacobi-Jordan algebra $ \left(J,[\cdot,\cdot],\alpha \right)  $.
\begin{defn} 
 An extension of a Hom-Jacobi-Jordan algebra $ J $ by $ V $ is an exact sequence
	$$0\longrightarrow \left( V,\rho,\beta\right)  \stackrel{i}{\longrightarrow}\left(  M,[\cdot,\cdot]_{M},\alpha_{M}\right)  \stackrel{\pi}{\longrightarrow } \left( J,[\cdot,\cdot],\alpha\right)  \longrightarrow 0$$
such that $ \pi $ is an homomorphisme of Hom-Jacobi-Jordan algebra, $\alpha_{M}\circ i=i\circ \beta  $ and $ \alpha\circ \pi=\pi\circ  \alpha_{M}$.\\
We say that the extension is 
abelian if $i(V)$ is an abelian ideal of M and 
split if 		
  there exists a subalgebra $ S\subset M$
		complementary to $ \ker \pi $.	
	Two extensions 
	\begin{displaymath}
		\xymatrix { 0 \ar[r] &  \left( V,\rho,\beta\right)\ar[d]_{id_{   V}}\stackrel{i}{\longrightarrow} &  \left(  M,[\cdot,\cdot]_{M},\alpha_{M}\right) \ar[d]_{\Phi} \stackrel{\pi}{\longrightarrow}& \left( J,[\cdot,\cdot],\alpha\right)\ar[r] \ar[d]_{id_J}&0 \\
			0\ar[r] &  \left( V,\rho,\beta\right)\stackrel{i'} {\longrightarrow}    &\left(  M,[\cdot,\cdot]'_{M},\alpha_{M}\right)  \stackrel{\pi'}{\longrightarrow} &\left( J,[\cdot,\cdot],\alpha\right)  \ar[r]&0}
	\end{displaymath} 
	are equivalent if there exist an  isomorphism  of Hom-Jacobi-Jordan algebra $\Phi:\left(  M,[\cdot,\cdot]_{M},\alpha_{M}\right)\rightarrow \left(  M,[\cdot,\cdot]'_{M},\alpha_{M}\right)$, 
	such that $\Phi \circ \ i= i'$ and $\pi'\ \circ\ \Phi= \pi.$			                           
\end{defn}
\begin{lem}	
	Let two equivalent  extensions
	\begin{gather*}
		\xymatrix {\relax 0 \ar[r] &( V,\rho,\beta )   \ar[d]_{Id_{V}}\relax\ar[r]^-{i_{0}} \relax & (J\oplus V,[\cdot,\cdot]_{\theta},\alpha+\beta ) \ar[d]_{\Phi} \ar[r]^-{\pi_{0}}& (J,[\cdot,\cdot],\alpha)\ar[r] \ar[d]_{Id_{J}}&0 \\
			\relax	0\ar[r] & ( V,\rho,\beta )  \ar[r]^-{i_{0}}     & (J\oplus V,[\cdot,\cdot]_{\theta'},\alpha+\beta ) \ar[r]^-{\pi_{0}} &(J,[\cdot,\cdot],\alpha)\ar[r]&0}
	\end{gather*}
	where $[\cdot,\cdot]_{\theta}=[\cdot,\cdot]+\rho+\theta  $ and $[\cdot,\cdot]_{\theta'}=[\cdot,\cdot]+\rho+\theta'  $. Then, $ \theta'$ is a $2$-cocycle of $J$ on $V$ and it is equivalent to $\theta'$ (i.e $(\theta-\theta')\in B^2_{\alpha,\theta}(J,V) $).
\end{lem}
\begin{proof}
We use the notation from Section  \ref{2group}:
\begin{align*}
\Phi(x+v)&=\pi_{J}(x)+\pi_{V}(x)+\pi_{J}(v)+ \pi_{V}(v)\\
&=s(x)+i(x)+s'(v)+ i'(v).	             
\end{align*}
 Since, $ \Phi\circ i_0=i_{0}  $ and $ \pi'_{0} \circ \Phi = \pi_0$.
  We have,
$v=i_{0}(v)=\Phi\circ i_0(v)=\Phi(v)=s'(v)+ i'(v)$. Hence, $ s'(v)=0 $ since $i'(v), v\in V$. Thus, $i'(v)=v$.   Moreover,	
$x=\pi_0(x)=\pi'_{0}(\Phi(x))=\pi'_{0}(s(x)+i(x))=s(x)$  since $\ker (\pi'_{0})=Im (i'_{0})$ and $ i(x)\in V $. Therefore, by \eqref{f}, we have $ f(x)=i(x) $. Thus, by $i'(v)=v$ and  \eqref{f1}
  	\begin{gather*}
\theta(x,y)-\theta'\left(x,y) \right)
=-f\left( [x,y]\right)+\rho\left(x\right)f(y)+\rho\left(y \right)f(x).	                
  \end{gather*}
Hence, $\theta'-\theta=d^1f  $.
\end{proof} 
\begin{lem}
	If $\theta$ and $\theta'$ are two equivalent $2$ cocycle of $J$ on $V$. Then, 
	the extensions $(J\oplus V,[\cdot,\cdot]_{\theta},\alpha+\beta )$ and $(J\oplus V,[\cdot,\cdot]_{\theta'},\alpha+\beta )$ are equivalent.
\end{lem}
\begin{proof}
	Let $\theta$ and $\theta'$ be two equivalent $2$ cocycle of $J$ on $V$. Then, there exist $h\in C_{\alpha,\beta}^1(J,V)$ satisfies $\theta'=\theta+d^1(h)$. Define the linear map $\Phi\colon J\oplus V\to J\oplus V$ by $\Phi(x+v)=x-h(x)+v$. Clearly, $\Phi$ is an isomorphism. Moreover, we have
	\begin{align*}
		&\left[\Phi(x+v),\Phi(y+w) \right]_{\theta'}\\
		&=d'\left[x-h(x)+v,y-h(y)+w \right]_{\theta'} \\
		&=[x,y] +\rho(x)w+\rho(y)v-	\rho(x)h(y)-
		\rho(y)h(x)+\theta'(x,y)\\
		&=[x,y] +\rho(x)w+\rho(y)v-	\rho(x)h(y)-
		\rho(y)h(x)+\theta(x,y)+d^1(h)(x,y)\\
	&	=[x,y] +\rho(x)w+\rho(y)v+\theta(x,y)-h\left([x,y] \right)\\
		&=\Phi\left( [x+v,y+w]_{\theta}\right) .
	\end{align*}
and $ \Phi\left((\alpha+\beta)(x+v) \right)=\alpha(x)-h(\alpha(x))+\beta(v) 
=\alpha(x)-\beta\left( h(x)\right) +\beta(v)=(\alpha+\beta)\left(\Phi (x+v) \right) $
	Hence, $\Phi\colon J\oplus V\to J\oplus V$ is an isomorphism of Hom-Jacobi-Jordan algebra. Therefore,                                  
	the following diagram 
	\begin{gather*}
		\xymatrix {\relax 0 \ar[r] &( V,\rho,\beta)   \ar[d]_{id_V}\relax\ar[r]^-{i_{0}} \relax & (J\oplus V,[\cdot,\cdot]_{\theta},\alpha+\beta ) \ar[d]_{\Phi} \ar[r]^-{\pi_{0}}& (J,[\cdot,\cdot],\alpha)\ar[r] \ar[d]_{id_J}&0 \\
			\relax	0\ar[r] & ( V,\rho,\beta ) \ar[r]^-{i_{0}}     & (J\oplus V,[\cdot,\cdot]_{\theta'},\alpha+\beta ) \ar[r]^-{\pi_{0}} &(J,[\cdot,\cdot],\alpha)\ar[r]&0}
	\end{gather*}
	gives, the extensions $(J\oplus V,[\cdot,\cdot]_{\theta},\alpha+\beta )$ and $(J\oplus V,[\cdot,\cdot]_{\theta'},\alpha )$ are equivalent.	
\end{proof}
\begin{lem}
	Let
	$$0\longrightarrow \left( V,\rho,\beta\right)  \stackrel{i}{\longrightarrow}\left(  M,[\cdot,\cdot]_{M},\alpha_{M}\right)  \stackrel{\pi}{\longrightarrow } \left( J,[\cdot,\cdot],\alpha\right)  \longrightarrow 0$$
be an abelian split extension of $J$ by $V$.	Then there exists an equivalent extension
  	$$0\longrightarrow \left( V,\rho,\beta\right)  \stackrel{i_{0}}{\longrightarrow}\left(  J\oplus V,[\cdot,\cdot]_{\theta},\alpha+\beta\right)  \stackrel{\pi_{0}}{\longrightarrow } \left( J,[\cdot,\cdot],\alpha\right)  \longrightarrow 0$$	         
\end{lem}
\begin{proof}
	Let 	$$0\longrightarrow \left( V,\rho,\beta\right)  \stackrel{i}{\longrightarrow}\left(  M,[\cdot,\cdot]_{M},\alpha_{M}\right)  \stackrel{\pi}{\longrightarrow } \left( J,[\cdot,\cdot],\alpha\right)  \longrightarrow 0$$
	be a split  
	extension of $ J $ by $ V $. Then there exist 
	a   subalgebra $ H\subset M$
	complementary to $ \ker \pi $. Since $Im\, i = \ker \pi$, we have    	
	$M= H\oplus i(V) $.   	   	
	The map  $\pi_{/H}:H\rightarrow J$ (resp $k:V\rightarrow i(V) $) defined by $\pi_{/H}(x)=\pi(x) $ (resp. $k( v)=i(v)$) is bijective, its inverse $ s  $ (resp. $l$).  Considering the map $\Phi\colon J\oplus V \to M$  defined  by
	$ \Phi(x+v)= s(x)+i (  v)$. Then,     $\Phi$ is an  isomorphism  and 
	$	M=s(J)\oplus i(V)$. Since 
 $s(J)$  and   $ i(V)$ are subalgebra of $M$, $ \alpha\circ \pi=\pi\circ  \alpha_{M}$,  $ \alpha_{M}\circ i=i\circ \beta  $ 
  we have $ \alpha_{M}=s\circ \alpha+i\circ \beta. $ Define a symmetric bilinear map $[\cdot,\cdot]_{\theta}\colon J\oplus V\to  J\oplus V $ by $[x+v,y+w]_{\theta}=\Phi^{-1}\left( [s(x)+i(v),s(y)+i(w)]_{M}\right)   $.
We show that $ \left(  J\oplus V,[\cdot,\cdot]_{\theta},\alpha+\beta\right)  $	is a Hom-Jordan-Jacobi algebra. We have
\begin{gather*}
0=\left[\alpha_{M}(s(x)+i(v)),[s(y)+i(w),s(z)+i(u)]_{M} \right]_{M} 
+\left[\alpha_{M}(s(y)+i(w)),[s(z)+i(u),s(x)+i(v)]_{M} \right]_{M}\\
+\left[\alpha_{M}(s(z)+i(u)),[s(x)+i(v),s(y)+i(w)] \right]_{M}\\
=\left[ s(\alpha(x))+i(\beta(v)),[s(y)+i(w),s(z)+i(u)]_{M}\right]_{M}
+\left[s(\alpha(y))+i(\beta(w)),[s(z)+i(u),s(x)+i(v)]_{M} \right]_{M}\\ 
+\left[s(\alpha(z))+i(u),[s(x)+i(v),s(y)+i(w)]_{M} \right]_{M}\\
=\left[ \Phi(\alpha(x)+\beta(v)),\Phi([y+w,z+u]_{\theta})\right]_{M}
+\left[\Phi(\alpha(y)+\beta(w)),\Phi([z+u,x+v]_{\theta}) \right]_{M}\\ 
+\left[\Phi(\alpha(z)+\beta(u)),\Phi([x+v,y+w]_{\theta}) \right]_{M}\\
=\Phi\Big(\left[ \alpha(x)+\beta(v),[y+w,z+u]_{\theta}\right]_{\theta}
+\left[\alpha(y)+\beta(w),[z+u,x+v]_{\theta} \right]_{\theta} 
+\left[\alpha(z)+\beta(u),[x+v,y+w]_{\theta} \right]_{\theta}            \Big).	            
\end{gather*}
Hence, $ \left(  J\oplus V,[\cdot,\cdot]_{\theta},\alpha+\beta\right)  $	is a Hom-Jacobi-Jordan algebra.
\end{proof}
According to the preceding lemmas, we get the following result.
\begin{thm}\label{bijection}
Let be $\left( V,\rho,\beta\right) $ be a representation of a  Hom-Jacobi-Jordan algebra $ \left(J,[\cdot,\cdot],\alpha \right)  $.	
 Then there is a bijective correspondence between the set 
 of equivalence classes of  split abelian extensions of $ J  $  by  $V$
 and the the  second  cohomology group of $ J $ on $V$.	                       
\end{thm}
\section{Classification of Low dimensional regular multiplicative solvable 
	Hom-Jordan-Jacobi algebras }
	Let $\left(M,[\cdot,\cdot]_{M},\alpha_{M} \right) $ be a  finite dimensional regular  multiplicative Hom-Jordan-Jacobi algebra (rm HJJ for short). As $ M $ is solvable, we have 
there exists an integer $k$ such that $ [D^{k}(M),D^{k}(M)]=D^{k+1}(M)=\{0\} $. Then $ D^{k}(M) $  is an abelian ideal of $M$. 
Let $ (v_1,\cdots,v_r) $ be a Jordan basis of $D^{k}(M)  $, 
and let $ (v_1,\cdots,v_r,u_1,\cdots, u_s )  $ be a Jordan basis of $M$. We set $ V=span\{v_1,\cdots,v_r   \} $, 
and $ J=span\{ u_1,\cdots, u_s     \} $. Then,
$ M=V\oplus J  $. Let $ \beta $  denote the restriction of $ \alpha_{M} $ to $ V $, $ \alpha_{J} $ denote the projection of $ \alpha_{M} $  to $ J  $, $ \theta $ denote the projection of $[\cdot,\cdot]_{M}  $ to $V$ and $ [\cdot,\cdot]_{J} $ denote the projection of $[\cdot,\cdot]_{M}  $ to $J$.
We use the above notations for the rest of paper.\\
Now, we rewrite some results appearing in the previous sections as follows.
\begin{thm}
	\label{structureMAY24}
	If $\alpha_{M}(J)\subset J  $,  we have
	\begin{enumerate}[(1)]
		\item $ \left( J, [\cdot,\cdot]_{J},\alpha \right)  $	 is a  Hom-Jacobi-Jordan algebra,
		\item the linear map $\rho\colon  J\to End(V)$, $x\longmapsto [x,\cdot]_{M}$ define a representation of $ J$ on $V$;
		\item $ \theta $ is a $2$-cocycle of the  Hom-Jacobi-Jordan algebra $  J $ 
		with coefficients in the representation $V$.
		\item Let $\theta'\colon J\times J \to V$ be a bilinear map. Define a symmetric bilinear map $[\cdot,\cdot]_{\theta'}$ on M by 	\[ [x+v,y+w]_{\theta'}=[x,y]_{J}+\rho(x)w +\rho(y)v+\theta'(x,y) .\] Then, $ \left(M,[\cdot,\cdot]_{\theta'},\alpha_{M} \right) $ becomes a Hom-Jacobi-Jordan algebra isomorphic to   
		$ \left(M,[\cdot,\cdot]_{M},\alpha_{M} \right) $	if and only if $\theta'-\theta\in B^2_{\alpha,\beta}(J,V)$.
	\end{enumerate}		               
\end{thm}	
\begin{proof}
	\begin{enumerate}[(1)]
		\item It is straightforward.
		\item See the preliminary of Definition    \ref{prilimenaryRep}.
		\item See the preliminary of Definition   \ref{pr2cocycle}.
		\item Remark that, if $ \theta'-\theta \in   B^2_{\alpha,\beta}(J,V)$ we have	$\theta'$ is a $2$ cocycle since $B^2_{\alpha,\beta}(J,V)$ is a vector space and $\theta.$ is a $2$ cocycle. For the rest see Theorem \ref{bijection}.
	\end{enumerate}				                
\end{proof}	
If $ \alpha_{M}(J)\subset J $, 
we apply the  the previous Theorem, to get an algorithm to construct all solvable regular multiplicative  HJJ
algebras of dimension $n$  given  algebras of dimension $ s<n $  in
following way:	
\begin{enumerate}[(i)]	
	\item Determine the linear maps  $\rho\colon J\to End(V)$  satisfied \eqref{rep1} and \eqref{rep2}.	
	\item Determine the $2$-cocycles, the $2$-coboundaries and compute the quotient
	$ H^2_{\alpha,\beta}(J,V) $.
	\item There exists a $\theta \in H^2_{\alpha,\beta}(J,V)$ such that $\left(M,[\cdot,\cdot]_{M},\alpha_{M} \right) $ is equivalent to  $\left(M,[\cdot,\cdot]_{\theta},\alpha_{M} \right) $. 
\end{enumerate}		
If $ \alpha_{M}(J)\nsubseteq J $. We use the following result.	
\begin{prop}\label{JordanBlok}
	Let $\left(M,[\cdot,\cdot]_M,\alpha_M  \right)$ be  a finite dimensional solvable multiplicative regular 
	Hom-Jordan-Jacobi algebra such that $\alpha_M  $ can be represented  by the matrix
	\[ 
	\begin{pmatrix}
		a & 1 &0 \cdots & 0 \\
		0 & a & 1\cdots & 0 \\
		\vdots & \vdots & \ddots & 1\\
		0 & 0 & \cdots & a
	\end{pmatrix} \] in a basis $ (v,u_1,\cdots,u_n) $.
	Then, $a=1$, $ [v,u_i ]_M=0  $ and $[u_i,u_j ]_M=x_{i,j}v  $.\\
	Conversely, if the above  conditions are satisfied,  we have $\left(M,[\cdot,\cdot]_M,\alpha_M  \right)$ is  a  finite dimensional solvable multiplicative regular 
	Hom-Jordan-Jacobi algebra.
\end{prop}
\begin{proof}
	It is clear that the only invariant subspace by $\alpha_{M} $ is $V=span\{v\}$. 
	Consequently, $D^1(M)=V  $
	since  $D^1(M)$ is an ideal of $M$.  
	The rest of proof is a straightforward.	
\end{proof}			 
\begin{prop}
	Each one-dimensional Hom-Jordan-Jacobi algebra  is abelian.	    
\end{prop}
\begin{proof}	
	The proof is a straightforward.	
\end{proof}			
Below we present the description of such Hom-Jordan-Jacobi algebras when dimension is equal to $2$.
\begin{prop}
	Every $2$-dimensional regular  multiplicative Hom-Jordan-Jacobi algebra is isomorphic to one of the following pairwise non-isomorphic Hom-Jordan-Jacobi algebra.
	\begin{enumerate}[$\bullet$]
		\item $J_{1,1}^{1}$: $  [e_1,e_1]=e_2 $, $ \alpha(e_1)=ae_1,\, \alpha(e_2)=a^2e_2 $,
		\item $J_{1,1}^{2}$: $  [e_2,e_2]=e_1 $, $ \alpha(e_1)=e_1,\, \alpha(e_2)=e_1+e_2 $.	
	\end{enumerate}	                      
\end{prop}
\begin{proof}
	Let $\left(M,[\cdot,\cdot]_M,\alpha_M  \right)$ be  a finite dimensional solvable multiplicative regular 
	Hom-Jordan-Jacobi algebra.	
	We consider two classes
	of morphism which are given by Jordan form, namely they are represented
	by the matrices
	$\begin{pmatrix}a&0\\0&b \end{pmatrix}$, 
	$\begin{pmatrix}a&1\\0&a \end{pmatrix}$.\\
	If $\alpha_{M}=\begin{pmatrix}
		a&0\\
		0&	b
	\end{pmatrix}$:	Then, $J$ is invariant under $\alpha_M  $. Hence we use the algorithm with $s=\dim (J)=1$. Let $ \rho(u_1)=xv_1 $. By $\rho([u_1,u_1])\beta(v_{1})=-2\rho(\alpha(u_1))\rho(u_{1})v_{1}$ we obtain $ x^2=0 $. Hence, $\rho=0$. Therefore, 
	\[H^2_{\alpha,\beta}(J,V)=C^2_{\alpha,\beta}(J,V) . \] Let  $\theta(u_1,u_1)=xv_1  $ be  a $2$-cohain of $J$ on $v$. Thus, $ \theta(\alpha(u_1),\alpha(u_1))=\beta\left( \theta(u_1,u_1)\right)  $.
	So, $ b^2=a $. We conclude that \[[\lambda_{1} v_1+\mu_{1} u_1,\lambda_{2} v_1+\mu_{2} u_1]_{\theta}=\mu_{1}\mu_{2}\theta(u_1,u_1)=\mu_{1}\mu_{2}xv_1. \]
	Making the change of basis: $e_1=\frac{1}{x}v_1  $, $ e_{2}=\frac{1}{x}u_1 $.Then, we obtain $ J_{1,1}^{1} $. \\
	If $\alpha_{M}=\begin{pmatrix}
		a&1\\
		0&	a
	\end{pmatrix}$: by Proposition \ref{JordanBlok}, we have  $ a=1 $, $ [v_{1},v_{1}]=[v_{1},u_1]=0 $, $ [u_1,u_1]=x_{1,1}v_{1} $. 
	Making the change of basis: $e_1=\frac{1}{x_{1,1}}v_1  $, $ e_{2}=\frac{1}{x_{1,1}}u_1 $. Then, we obtain $ J_{1,1}^{2} $. \\
\end{proof}
\begin{lem}\label{lemma2DIAGONAL}
	Let $\left(M,[\cdot,\cdot]_M,\alpha_M  \right)$ be  an abelian extension of $ J_{1,1}^{1}$ by $V$ such that $\dim(V)=1$. Then, $ \left(M,[\cdot,\cdot]_M,\alpha_M  \right) $ is isomorphic to 	  $\left(M,[\cdot,\cdot]_0,\alpha_M\right) $ where  $ [u_1,u_1]_0=u_2 $ (the other brackets are zero).               
\end{lem}
\begin{proof}
	Let $ J= J_{1,1}^{1}$.	Then, we have $ [u_1,u_1]_{J}=u_{2} $, $ \alpha(u_1)=au_1 $ and $ \alpha(u_2)=a^2u_2 $. Let $ \beta (v_1)=bv_1 $,$ \rho(u_1)(v_1)=x_1v_1 $ and $ \rho(u_2)(v_1) =x_2 v_1$.
	By \eqref{rep2}, we have $ \rho([u_1,u_1]_{J})=-2\rho(\alpha(u_1))\rho(u_1) $ and $ \rho([u_2,u_2]_{J})=-2\rho(\alpha(u_2))\rho(u_2) $.
	Hence, $ \rho=0 $.
	The task is now to find the space $Z_{\alpha,\beta}^{2}(J,V)$. Let $\theta$ a $2$-cocycle of $J$ on $V$. Then $\theta$ satisfies the Eq.
	\eqref{adjoint2}. Hence, $ \theta(\alpha(u_1),[u_1,u_1]) =0$ and $ \theta(\alpha(u_2),[u_1,u_1]) =0$. Thus, $ \theta(u_1,u_2) =0$ and $ \theta(u_2,u_2) =0 $.\\
	Next, we will find the space $B_{\alpha,\beta}^{2}(J,V)$. Let $g$ be a $2$ cobord of $J$ on $V$. By \eqref{adjoint1}, we obtain $ g(u_1,u_1)=f([u_1,u_1]) =f(u_2)$, $ g(u_1,u_2)=0$ and  $ g(u_2,u_2)=0$.
	Hence, any $2$-cocycle is a $2$-cobord. Then $H^2(J,V)=\{0\}$ and   by Theorem \ref{structureMAY24}, $ [u_1,u_1]_{0}=
	[u_1,u_1]_{J}+\theta(u_1,u_1)=u_2 $ and  
	$ \left(M,[\cdot,\cdot]_{M},\alpha_{M} \right)  $ is isomorphic  to $\left(M,[\cdot,\cdot]_0,\alpha_M\right) $.
\end{proof}
\begin{lem}\label{lemma2Jorfanblock}
	Let $\left(M,[\cdot,\cdot]_M,\alpha_M  \right)$ be  an abelian extension of $ J_{1,1}^{2}$ by $V$ such that $\dim(V)=1$. Then, 
	$ \alpha_M(u_1)=u_1 $, $ \alpha_M(u_2)=u_1+u_2 $, $ \alpha_M(v_1)=bv_1 $
	$ \left(M,[\cdot,\cdot]_M,\alpha_M  \right) $ is isomorphic to 	  $\left(M,[\cdot,\cdot]_0,\alpha_M\right) $ where  $ [u_2,u_2]_0=u_1 $ (the other brackets are zero).               
\end{lem}	           
\begin{proof}
	This proof is similar to the previous Lemma.	           
\end{proof}
\begin{lem}\label{abelian}
	Let	$\left( J,[\cdot,\cdot]_{J},\alpha_{J}\right) $ be an  abelian HJJ of dimension $2$ such that $ \alpha_{J}(u_1)=au_1 $, $ \alpha_{J}(u_2)=bu_2 $. 
	Let $\left(M,[\cdot,\cdot]_M,\alpha_M  \right)$ be  an abelian extension of $ J$ by $V$ such that $\dim(V)=1$ and $\beta(v_1)=dv_1$. Then, 
	\[ [u_{1},u_{1}]_{M}=x_{1}v_{1},\,[u_{1},u_{2}]_{M}=x_{2}v_{1},\,[u_{2},u_{2}]_{M}=x_{3}v_{1} \] 
	where
	\[\left\{
	\begin{array}{ll}
		x_{2}&=0 \text{ if  }  b=-a, \,  d=a^2\\
		x_{2}&=0,\, x_{3}=0  \text{ if  } b^2\neq a^2,\,  d=a^2.
	\end{array}\right.\]             
\end{lem}
\begin{proof}
	By straightforward computations, we have $\rho=0  $, $Z^2_{\alpha,\beta}(J,V)=C^2_{\alpha,\beta}(J,V)  $ and  $ H^2_{\alpha,\beta}(J,V)=C^2_{\alpha,\beta}(J,V) $. Let $\theta$ be a $2$-cocycle defined by $\theta(u_1,u_1)=x_{1}v_{1}  $, $\theta(u_1,u_2)=x_{2}v_{1}  $, $\theta(u_2,u_2)=x_{3}v_{1}  $. Then, The condition $ \theta\circ \alpha=\beta \circ \alpha $ is equivalent to the system
	\[\left\{
	\begin{array}{ll}
		x_{1}(d-a^2)&=0\\
		x_{2}(d-ab)&=0\\
		x_{3}(d-b^2)&=0 
	\end{array}\right.\]
	Therefore,  if $d=a^2$ and $b=- a$, we have $\theta(u_1,u_2)=0  $.\\
	If $d=a^2$ and  $a^2\neq b^2$, we have $\theta(u_2,u_2)=0  $, $\theta(u_1,u_2)=0  $.\\ 
	If $d\neq a^2$ , we have $\theta(u_1,u_1)=0  $, $\theta(u_1,u_2)=0  $. 
\end{proof}
\begin{lem}\label{JordanAbelian}
	Let	$\left( J,[\cdot,\cdot]_{J},\alpha_{J}\right) $ be an  abelian HJJ of dimension $2$ such that $ \alpha_{J}(u_1)=au_1 $, $ \alpha_{J}(u_2)=u_1+au_2 $ . 
	Let $\left(M,[\cdot,\cdot]_M,\alpha_M  \right)$ be  an abelian extension of $ J$ by $V$ such that $\dim(V)=1$ and $\beta(v_1)=a^2v_1$. Then, 
	\[ [u_{1},u_{1}]_{M}=0,\,[u_{1},u_{2}]_{M}=0,\,[u_{2},u_{2}]_{M}=x_{3}v_{1} \] 		               
\end{lem}
\begin{proof}
	Let $\beta(v_1)=cv_1$. Similar to the previous Lemma we obtain
	\[\left\{
	\begin{array}{ll}
		x_{1}(c-a^2)&=0\\	
		ax_{1}+x_{2}(c-a^2)&=0\\
		x_1+2ax_2+	x_{3}(a^2-c)&=0 
	\end{array}\right.\] 
	Hence, if $c=a^2$, we obtain $x_1=x_2=0$ and if $c\neq a^2$ we obtain $x_1=x_2=x_3=0  $ 
\end{proof}
\begin{lem} \label{Jvariant}
	Let $\left(M,[\cdot,\cdot]_M,\alpha_M  \right)$ be  a  $3$-dimensional solvable multiplicative regular 
	Hom-Jordan-Jacobi algebra such that   $\alpha_{M}(J)\nsubseteq J$. Then, $\dim(V)=1$ and the structures of $ M $ are given  by the matrix 
	$\begin{pmatrix}
		1&1&0\\
		0&	1&1\\
		0&	0&1
	\end{pmatrix}$
	with respect to the basis $ (v_1,u_1,u_2) $ and $ [v_{1},u_1]=[v_{1},u_2]=0 $, $ [u_1,u_1]=x_{1,1}v_{1} $, $ [u_1,u_2]=x_{1,2}v_{1} $ and
	$ [u_2,u_2]=x_{2,2}v_{1} $.
\end{lem}
\begin{proof}
	Since $\alpha_{M}(J)\nsubseteq J$, the matrix of  $\alpha_{M}$ has the following form: $\begin{pmatrix}
		a&1&0\\
		0&	a&1\\
		0&	0&a
	\end{pmatrix}$. Therefore, from Lemma\ref{JordanBlok}  we finish the proof.  	     
\end{proof}

\begin{thm}
	Every $3$-dimensional regular  multiplicative Hom-Jacobi-Jordan algebra is isomorphic to one of the following pairwise non-isomorphic Hom-Jordan-Jacobi algebra.
	\begin{enumerate}[$\bullet$]
		\item $ J^1_{2,1} $: $[e_1,e_1]=e_2  $, 
		$ \alpha(e_1)=ae_1 $,  $ \alpha(e_2)=a^2e_2 $,  $ \alpha(e_3)=be_3 $,
		\item $ J^2_{2,1} $: $[e_1,e_1]=e_3  $,  $[e_2,e_2]=e_3  $,
		$ \alpha(e_1)=ae_1 $,  $ \alpha(e_2)=-ae_2 $,  $ \alpha(e_3)=a^2e_3 $,
		\item $ J^3_{2,1} $: $[e_1,e_1]=e_3  $,  
		$ \alpha(e_1)=ae_1 $,  $ \alpha(e_2)=-ae_2 $,  $ \alpha(e_3)=a^2e_3 $,
		\item $ J^4_{2,1} $: $[e_1,e_1]=e_3  $,  
		$ \alpha(e_1)=ae_1 $,  $ \alpha(e_2)=be_2 $ $(b^2\neq a^2)$,  $ \alpha(e_3)=a^2e_3 $,
		\item $ J^5_{2,1} $: $[e_2,e_2]=e_1  $,
		$ \alpha(e_1)=e_1 $,  $ \alpha(e_2)=e_1+e_2 $   $\alpha(e_3)=be_3 $,
		\item $ J^6_{2,1} $: $[e_2,e_2]=e_3  $,
		$ \alpha(e_1)=ae_1 $,  $ \alpha(e_2)=ce_1+ae_2 $,  $ \alpha(e_3)=a^2e_3 $,
		\item $ J^7_{1,2} $: $[e_1,e_1]=e_2  $,
		$ \alpha(e_1)=ae_1 $,  $ \alpha(e_2)=a^2e_2 $,  $ \alpha(e_3)=ce_3 $ ($ c\neq a^2 $),
		\item $ J^8_{1,2} $: $[e_1,e_1]=xe_1+ye_2  $,
		$ \alpha(e_1)=ae_1 $,  $ \alpha(e_2)=a^2e_2 $,  $ \alpha(e_3)=a^2e_3 $ ,
		\item $ J^9_{1,2} $: $[e_1,e_1]=e_2  $,
		$ \alpha(e_1)=ae_1 $,  $ \alpha(e_2)=a^2e_2 $,  $ \alpha(e_3)=ce_2+a^2e_3 $ ,
		\item $ J^{10}_{1,2} $: $[e_1,e_3]=e_2  $,
		$ \alpha(e_1)=ae_1 $,  $ \alpha(e_2)=a^2e_2 $,  $ \alpha(e_3)=a^2e_3 $
		\item $ J^{11}_{1,2} $: $[e_1,e_3]=e_2  $,
		$ \alpha(e_1)=ae_1 $,  $ \alpha(e_2)=a^2e_2 $,  $ \alpha(e_3)=ce_2+a^2e_3 $,
		\item $ J^{12}_{2,1} $: $[e_1,e_1]=e_3  $, $[e_1,e_2]=e_3  $, $[e_2,e_2]=xe_3  $,
		$ \alpha(e_1)=e_1+e_3 $,  $ \alpha(e_2)=ce_1+e_2 $,  $ \alpha(e_3)=e_3 $, 
		\item $ J^{13}_{2,1} $: $[e_1,e_1]=e_3  $, $[e_2,e_2]=e_3  $,
		$ \alpha(e_1)=e_1+e_3 $,  $ \alpha(e_2)=ce_1+e_2 $,  $ \alpha(e_3)=e_3 $, 
		\item $ J^{14}_{2,1} $: $[e_1,e_1]=e_3  $, 
		$ \alpha(e_1)=e_1+e_3 $,  $ \alpha(e_2)=e_1+e_2 $,  $ \alpha(e_3)=e_3 $, 
		\item $ J^{15}_{2,1} $: $[e_1,e_2]=e_3  $, $[e_2,e_2]=e_3  $, 
		$ \alpha(e_1)=e_1+e_3 $,  $ \alpha(e_2)=ce_1+e_2 $,  $ \alpha(e_3)=e_3 $,
		\item $ J^{16}_{2,1} $:  $[e_2,e_2]=e_3  $, 
		$ \alpha(e_1)=e_1+e_3 $,  $ \alpha(e_2)=ce_1+e_2 $,  $ \alpha(e_3)=e_3 $,
		\item $ J^{17}_{2,1} $:  $[e_1,e_2]=e_3  $, 
		$ \alpha(e_1)=e_1+e_3 $,  $ \alpha(e_2)=ce_1+e_2 $,  $ \alpha(e_3)=e_3 $,
	\end{enumerate}		                   
\end{thm}
\begin{proof}
	\begin{Cases}  $\dim(V)=1 $ and $ \alpha_M(J)\subset J $.\\
		Since $ J $ is a $2$-dimensional rm HJJ algebra. Then,
		there is a basis $\{v,u_1,u_2\}$ 
		such
		that the matrix representing $\alpha_{M}  $ with respect to this basis is		
		by one of the following forms:
		$\begin{pmatrix}
			c&0&0\\
			0&	a&0\\
			0&	0&b
		\end{pmatrix}$, $\begin{pmatrix}
			c&0&0\\
			0&	a&1\\
			0&	0&a
		\end{pmatrix}$. We recall that $ V=span\{v\} $ and $J=span \{u_1,u_2 \} $.
		By Theorem \ref{structureMAY24}, $V$ is representation of $J$ since 
		$ V $ is an abelian ideal of $M$. We set $\beta(v)=b  v $, $\rho(u_1)v=x v  $, $\rho(u_2)v=yv$.\\ 
		\begin{scass}{} If $\alpha_{M}=\begin{pmatrix}
				c&0&0\\
				0&	a&0\\
				0&	0&b
			\end{pmatrix}$. Then, $J$ is isomorphic to $J_{1,1}^{1}$ or it is abelian. 
			\begin{enumerate}[(i)]
				\item  $J$ is isomorphic to $J_{1,1}^{1}$:	
				According to Lemma \ref{lemma2DIAGONAL}, $ \left(M,[\cdot,\cdot]_M,\alpha_M  \right) $ is isomorphic to 	  $\left(M,[\cdot,\cdot]_0,\alpha_M\right) $ where  $ [u_1,u_1]_0=u_2 $. If we take $e_1=u_1$, $e_2=u_2 $ and $v_1=e_3$ 
				we obtain $J^{1}_{2,1}$.
				\item  $J$ is abelian: 	According to Lemma \ref{abelian}, we have the following cases:
				\[c=a^2 \text{  and  }  b=-a :\left\{
				\begin{array}{ll}
					x_1\neq 0  \text{  and  } x_3\neq 0	 &: \left[\frac{u_{1}}{x_{1}},\frac{u_{1}}{x_{1}} \right]=\frac{v_{1}}{x_{1}},\,
					\left[\frac{u_{2}}{\sqrt{x_{1}x_{1}}},\frac{u_{2}}{\sqrt{x_{1}x_{1}}} \right]=\frac{v_{1}}{x_{1}}   \\
					x_1\neq 0  \text{  and  } x_3= 0		&: \left[\frac{u_{1}}{x_{1}},\frac{u_{1}}{x_{1}} \right]=\frac{v_{1}}{x_{1}}\\
					x_1 =0  \text{  and  } x_3\neq 0		&: \left[\frac{u_{2}}{x_{3}},\frac{u_{2}}{x_{3}} \right]=\frac{v_{1}}{x_{3}}
				\end{array}\right.\]
				For case $x_1\neq 0$    and   $x_3\neq 0$: we take $ e_1 =\frac{u_{1}}{x_{1}}$, $ e_2=\frac{u_{2}}{\sqrt{x_{1}x_{1}}} $, $ e_3=\frac{v_{1}}{x_{1}} $. Then, $ J^{2}_{2,1} $.\\
				For case $x_1\neq 0$    and   $x_3= 0$: we take $ e_1 =\frac{u_{1}}{x_{1}}$, $ e_2=u_{2} $, $ e_3=\frac{v_{1}}{x_{1}} $. Then, $ J^{3}_{2,1} $.\\
				For case $x_1= 0$    and   $x_3= 0$: we take $ e_1 =\frac{u_{2}}{x_{3}}$, $ e_1=u_{2} $, $ e_3=\frac{v_{1}}{x_{3}} $, $ c=-a $. Then, we obtain $ M $ is also in the family $ J^{3}_{2,1} $.\\
				If $ c=a^2$   and  $b^2\neq a^2  $: $ x_{2}=x_3=0 $.
				we take $ e_1 =\frac{u_{1}}{x_{1}}$, $ e_2=u_{2} $, $ e_3=\frac{v_{1}}{x_{1}} $. Then, $ J^{4}_{2,1} $.\\	
			\end{enumerate}
		\end{scass}	
		\begin{scass}{} 	If $\alpha_{M}=\begin{pmatrix}
				b&0&0\\
				0&	a&1\\
				0&	0&a
			\end{pmatrix}$, 
			$J$ is isomorphic to $J_{1,1}^{2}$ (since $ \alpha $ is not diagonalizable) or it is abelian.
			\begin{enumerate}[(i)]
				\item 	If  $ J $ not abelian. Then, $J$ is isomorphic to $J_{1,1}^{2}$. Thus, using Lemma\ref{lemma2Jorfanblock} $M$ is isomorphic to 
				$ J^5_{2,1} $.
				\item If $J$ is abelian. Then, 	using Lemma \ref{JordanAbelian}, $M$ is isomorphic to 	$ J^6_{2,1} $.	 
			\end{enumerate}	 
		\end{scass}	
	\end{Cases}
	
	\begin{Cases}  $\dim(V)=2 $ and $ \alpha_M(J)\subset J $.\\
		There is a basis $\{v_{1},v_2,u_1\}$ 
		such
		that the matrix representing $\alpha_{M}  $ with respect to this basis is		
		by one of the following forms:
		$\begin{pmatrix}
			b&0&0\\
			0&	c&0\\
			0&	0&a
		\end{pmatrix}$, $\begin{pmatrix}
			b&1&0\\
			0&	b&0\\
			0&	0&a
		\end{pmatrix}$. We recall that $ V $ is abelian.
		By Theorem \ref{structureMAY24}, $V$ is representation of $J$ since 
		$ V $ is an abelian ideal of $M$. Let $ A=\begin{pmatrix}
			x_1&y_1\\
			x_2&	y_2
		\end{pmatrix} $  the matrix of $\rho(u_1) $ related with the basis $ (v_1,v_2) $ of $V$. By  $ J $ is abelian, $ \alpha_{M}(u_1)=au_1 $ and
		\eqref{rep2}, we have  $ A^2=0  $.  The solutions of the equation $A^2=0  $ are 
		$ \begin{pmatrix}
			0&0\\
			0&0
		\end{pmatrix}$ and $ \begin{pmatrix}
			x_1&-\frac{ x_1^2}{x_2}\\
			x_2&	-x_1
		\end{pmatrix}$. 
		\begin{scass}  $A= \begin{pmatrix}
				0&0\\
				0&0
			\end{pmatrix}$ and $\alpha_{M}=\begin{pmatrix}
				b&0&0\\
				0&	c&0\\
				0&	0&a
			\end{pmatrix}$.\\
			Then, 
			$ H^2_{\alpha,\beta}(J,V)=C^2_{\alpha,\beta}(J,V) $ since $\rho=0$ and $[\cdot,\cdot]=0$.
			Let  $\theta(u_1,u_1)=xv_1+yv_2  $ be  a $2$-cohain of $J$ on $v$. \\
			The condition $ \theta\circ \alpha=\beta \circ \alpha $ is equivalent to the system
			\[\left\{
			\begin{array}{ll}
				x(b-a^2)&=0\\
				y(a^2-c)&=0 
			\end{array}\right.\]
			Hence: if $ b=a^2 $, $ c\neq a^2 $, we have $ \theta(u_1,u_1)=xv_1 $ and $M$ is isomorphic to $ J^7_{1,2} $.\\
			If $ c=a^2 $, $ b\neq a^2 $, we have $ \theta(u_1,u_1)=yv_2 $. Hence, with the basis $ (v_2,v_1,u) $ we obtain the previous case. Then, $ J^7_{1,2} $.\\
			If $ c=a^2 $, $ b= a^2 $,
			we have  $M$ is isomorphic to $ J^8_{1,2} $.
		\end{scass}	
		\begin{scass}  $A= \begin{pmatrix}
				0&0\\
				0&0
			\end{pmatrix}$ and	
			\mathversion{bold}
			If $\alpha_{M}=\begin{pmatrix}
				b&1&0\\
				0&	b&0\\
				0&	0&a
			\end{pmatrix}$.
			\mathversion{normal}
			Since $\theta\circ \alpha=\beta \circ \alpha $, one has
			\[\left\{
			\begin{array}{ll}
				x(b-a^2)&=y	\\
				y(a^2-b)&=0 
			\end{array}\right.\]
			As a consequence, we obtain $2$-coycles as follows:
			\begin{enumerate}[(i)]
				\item $ \theta(u_1,u_1)=xv_1 $ if $ b=a^2 $,
				\item $ \theta(u_1,u_1)=0 $ if $ b\neq a^2 $.	
			\end{enumerate} 
			The first $2$-coycle
			correspond to anyone rm HJJ algebra $ J^{9}_{1,2} $ and  the second to an abelian   HJJ algebra.	                 
		\end{scass}
		\begin{scass} 
			$ A=\begin{pmatrix}
				x_1&-\frac{ x_1^2}{x_2}\\
				x_2&	-x_1
			\end{pmatrix}$   and $\alpha_{M}=\begin{pmatrix}
				b&0&0\\
				0&	c&0\\
				0&	0&a
			\end{pmatrix}$.\\. Let $\theta$ be a $2$-cocyle of $J$ on $V$. Then, by \eqref{cocycle}, we have $ \rho(\alpha(u_1))\theta(u_1,u_1) =0$. Which implies that $\theta(u_1,u_1)=y\left( \frac{ x_1}{x_2}v_1  +v_2 \right) $ since $ \alpha(u_1)=au_1 $ and $ \rho(u_1)=A $.
			The condition $\theta\circ \alpha=\beta \circ \alpha $ gives $b=c=a^2$
			As a consequence, we obtain  HJJ  algebras as follow:
			$ [u_1,v_1]=x_1v_1+x_2v_2 $, $[u_1,v_2]=-x_1\left(\frac{ x_1}{x_2} v_1+v_2 \right) $, $ [u_1,u_1]=y\left( \frac{ x_1}{x_2}v_1  +v_2 \right) $. Let $w=x_1v_1+x_2v_2$. In the basis $ (v_1,w,u_1) $ we have 
			$\alpha_{M}=\begin{pmatrix}
				a^2&0&0\\
				0&	a^2&0\\
				0&	0&a
			\end{pmatrix}$ and $ [u_1,v_1]=w$, $ [u_1,w] =0$, $ [u_1,u_1]=\frac{y}{x_2}w$.   
			Let us consider the isomorphism $f\colon M\to M$ defined by
			by $ f(v_1)=v_1 $, $ f(w)=\frac{y}{x_2} w $,  $ f(u_1)=\frac{y}{x_2} u_1 $. We obtain  $ J^{10}_{1,2} $.
		\end{scass}	
		\begin{scass} 
			$ A=\begin{pmatrix}
				x_1&-\frac{ x_1^2}{x_2}\\
				x_2&	-x_1
			\end{pmatrix}$   and $\alpha_{M}=\begin{pmatrix}
				b&1&0\\
				0&	b&0\\
				0&	0&a
			\end{pmatrix}$.\\. Let $\theta$ be a $2$-cocyle of $J$ on $V$. Then, by \eqref{cocycle}, we have $ \rho(\alpha(u_1))\theta(u_1,u_1) =0$. Which implies that $\theta(u_1,u_1)=y\left( \frac{ x_1}{x_2}v_1  +v_2 \right) $ since $ \alpha(u_1)=au_1 $ and $ \rho(u_1)=A $.
			The condition $\theta\circ \alpha=\beta \circ \alpha $ gives $b=c=a^2$
			As a consequence, we obtain  HJJ  algebras as follow:
			$ [u_1,v_1]=x_1v_1+x_2v_2 $, $[u_1,v_2]=-x_1\left(\frac{ x_1}{x_2} v_1+v_2 \right) $, $ [u_1,u_1]=y\left( \frac{ x_1}{x_2}v_1  +v_2 \right) $. Let $w=x_1v_1+x_2v_2$. In the basis $ (v_1,w,u_1) $ we have 
			$\alpha_{M}=\begin{pmatrix}
				a^2&x_2&0\\
				0&	a^2&0\\
				0&	0&a
			\end{pmatrix}$ and $ [u_1,v_1]=w$, $ [u_1,w] =0$, $ [u_1,u_1]=\frac{y}{x_2}w$.   
			Let us consider the isomorphism $f\colon M\to M$ defined by
			by $ f(v_1)=v_1 $, $ f(w)=\frac{y}{x_2} w $,  $ f(u_1)=\frac{y}{x_2} u_1 $. We obtain  $ J^{11}_{1,2} $.
		\end{scass}			
	\end{Cases}
	\begin{Cases}   $ \alpha_M(J) \nsubseteq J $.Then, from lemma \ref{Jvariant}:
		If $x_{1,1}\neq 0  $ and  $x_{1,2}\neq 0  $ we have
		$ [\frac{u_1}{x_{1,1}},\frac{u_1}{x_{1,1}}]=\frac{v_{1}}{x_{1,1}} $, $ [\frac{u_1}{x_{1,1}},\frac{u_2}{x_{1,2}}]=\frac{v_{1}}{x_{1,1}} $ and
		$ [\frac{u_2}{x_{1,2}},\frac{u_2}{x_{1,2}}]=\frac{x_{2,2}x_{1,1}}{x_{1,2}^2}\frac{v_{1}}{x_{1,1}} $. Hence, if we take $ e_1=\frac{u_1}{x_{1,1}} $, $ e_2=\frac{u_2}{x_{1,2}} $ and $ e_3=\frac{v_{1}}{x_{1,1}} $	 we obtain   $ J^{12}_{1,2} $.\\
		Similar, if  $x_{1,1}\neq 0  $,  $x_{1,2}= 0  $ and  $x_{2,2}\neq 0  $
		we obtain  $J^{13}_{1,2} $.\\
		if  $x_{1,1}\neq 0  $,  $x_{1,2}= 0  $ and  $x_{2,2}= 0  $
		we obtain  $J^{14}_{1,2} $.\\
		
	\end{Cases}	
\end{proof}
\section{ Metric  Hom-Jacobi-Jordan algebras}
In this section, we introduce the notion of metric  Hom-Jacobi-Jordan algebras,  and provide some properties.	
\begin{defn}
	A metric  Hom-Jacobi-Jordan algebra  is a $ 4 $-tuple $ \left( J,[\cdot,\cdot],\alpha,B \right)  $ consisting of a Hom-Jacobi-Jordan algebra $ \left( J,[\cdot,\cdot],\alpha \right)  $ and a nondegnerate symmetric bilinear form $B$
	satisfying :
	\footnotesize{
		\begin{align}
			B(x,[y,z]))&=B([x,y],z) \text{  (invariance of $B$)}\label{invDef} \\  B(\alpha(x),y)&=B(x,\alpha(y))\text{  (Hom-invariance of $B$)}\label{alphadef}		      
	\end{align}}	
	for any $ x,\,y,\, z\in J $.
	for any $ x,\,y,\, z\in J $.
	We recover the  metric Jacobi-Jordan algebra   when
	$ \alpha = id$.
\end{defn}
We say that two metric Hom-Jacobi-Jordan algebras  $ \left( J,[\cdot,\cdot],\alpha,B \right)   $  and $ \left( J',[\cdot,\cdot]',\alpha',B' \right)   $ are isometrically
isomorphic (or $i$-isomorphic, for short) if there exists a  Hom-Jacobi-Jordan isomorphism $ f $
from $ J $ onto $ J' $ satisfying $ B'\left(f(x),f(y) \right)=B(x,y)  $
for all $ x,y\in J $ In this case,
$ f $ is called an i-isomorphism.\\
\begin{defn}
	Let $I$ be an ideal of	a metric Hom-Jacobi-Jordan algebra		 $ \left( J,[\cdot,\cdot],\alpha,B \right)   $.   		                      
	\begin{enumerate}[(1)]
		\item  The orthogonal $ I^{\perp} $ of I, with respect to $B$, is defined by
		\[  I^{\perp}=\{x\in \mathfrak J\mid B( x,y)=0\, \forall y\in I  \}. \]
		\item An ideal $I$ is isotropic if $I\subset I^{\perp}$.	
	\end{enumerate}	
\end{defn}

Let $( J,[\cdot,\cdot],\alpha,B)$  be a multiplicative metric  Hom-Jacobi-Jordan algebra. Since $B$ is non-degenerate and invariant
we have some simple properties of $J$ as follows:				
\begin{prop}
	\begin{enumerate}[(1)]	
		\item 	the center $\mathfrak Z(J)$ is an ideal of $J$.
		\item $\mathfrak Z(J)=[J,J]^\perp$ and then $ \dim(\mathfrak Z(J))+\dim([J,J])=\dim(J)$. 
	\end{enumerate}			
\end{prop}
\begin{prop}\label{IdealorthogonalCentre}
	Let $I$ be an ideal of a metric Hom-Jacobi-Jordan algebra $( J,[\cdot,\cdot],\alpha,B)$. Then,
	\begin{enumerate}[(1)]					
		\item $I^{\perp}$  is an ideal of $J$.
		\item   the centralizer $ \mathfrak Z(I) $ of $I  $ contains $ I^{\perp}$.
	\end{enumerate}			
\end{prop}
For the rest of this paper, for any metric Hom-Jacobi-Jordan algebra, the generalized coadjoint representation identity it satisfied.
\begin{prop}	\label{gamma}
	A $4$-tuple $ \left(  J,[\cdot,\cdot],\alpha,B \right) $ is a metric Hom-Jacobi-Jordan algebra if and only if $B$ is a 	nondegenerate  symmetric bilinear form satisfying  $B(\alpha(x),y)=B(x,\alpha(y))  $ and 	$d_{r}^3\gamma=0$ where 
	$\gamma (x,y,z)= B([x,y],z) $ and 	$d_{r}^3  $ is given in \eqref{operatorCoadjoint3}.
\end{prop}
\begin{proof}
	Let $B$ a 	 nondegenerate  symmetric bilinear form such that for all $x,y\in J$,
	\begin{equation}\label{alphaInvariant}
		B(\alpha(x),y)=B(x,\alpha(y)). 	                         
	\end{equation}	
	For all $x,y,z\in J$, we have
	\begin{align}
		d_{r}^3&\gamma(x,y,z,t)\nonumber\\
		&=
		\gamma([x,y],\alpha(z),t)
		+\gamma([x,z],\alpha(y),t)
		+\gamma([y,z],\alpha(x),t)\nonumber\\	
		+& \gamma\left( x,y,[\alpha(z),t ]\right)
		+\gamma\left( y,z,[\alpha(x),t] \right)
		+\gamma\left( x,z,[\alpha(y),t] \right)\nonumber\\
		&=B(\left[ [x,y],\alpha(z)\right] ,t)+B\left( \left[ [x,z],\alpha(y)\right] ,t\right) 
		+B\left( \left[ \alpha(x),[y,z]\right] ,t\right) \label{invariant1}\\
		+&B([x,y],[\alpha(z),t] )
		+B([y,z],[\alpha(x),t] )
		+B([x,z],[\alpha(y),t]).\label{invariant2}		
	\end{align}
	If the condition \eqref{invDef} is satisfied, we have
	\begin{align*}
		\eqref{invariant2}=B(x,\left[ y,[\alpha(z),t]\right]  )
		+B(\left[ [y,z],t\right] ,\alpha(x) )
		+B(x,\left[ z,[\alpha(y),t]\right] )
	\end{align*}
	By 	\eqref{alphaInvariant}, we have 
	$B(\left[ [y,z],t\right] ,\alpha(x) )=B\left(\alpha\left(\left[ [y,z],t\right] \right),x  \right)   $. Hence,
	\begin{equation*}
		\eqref{invariant2}=B(x,\left[ y,[\alpha(z),t]\right]  )
		+B(x, \alpha\left(\left[ [y,z],t\right] \right))
		+B(x,\left[ z,[\alpha(y),t]\right] ) \label{hbel}	     
	\end{equation*}
	Then, if \eqref{invDef} and \eqref{alphaInvariant} are satisfied, we obtain
	\begin{align}
		d_{r}^3&\gamma(x,y,z,t)\nonumber\\
		&=B(\left[ [x,y],\alpha(z)\right] ,t)+B\left( \left[ [x,z],\alpha(y)\right] ,t\right) 
		+B\left( \left[ \alpha(x),[y,z]\right] ,t\right) \label{Oinvariant1}\\
		+&B(x,\left[ y,[\alpha(z),t]\right]  )
		+B(x, \alpha\left(\left[ [y,z],t\right] \right))
		+B(x,\left[ z,[\alpha(y),t]\right] )\label{Oinvariant2}
	\end{align}		
	By the Hom-Jacobi identity, we obtain that \eqref{Oinvariant1}=0. \\ On the other hand,
	by the generalized coadjoint representation identity, we obtain  \eqref{Oinvariant2}=0.\\			
	Therefore, $d_r^3\gamma=0$.\\
	We now show that $ \gamma\in  S^3(J,\R)  $. For all $x,y,z\in J$,
	by Eq.\eqref{invDef}, $[\cdot,\cdot]$ and $B$ are symmetric, we have 
	\begin{equation*}
		B([x,y],z)=B([y,x],z)=B\left(y,[x,z] \right)=B\left([x,z] ,y \right),  
	\end{equation*}	
	which implies that
	\begin{equation*}
		\gamma (x,y,z)=\gamma (y,x,z) =\gamma (x,z,y).                                         
	\end{equation*}
	So
	\begin{equation*}
		\gamma (x,z,y)=\gamma (z,x,y)		
		=\gamma (x,y,z)                                        
	\end{equation*}
	and 
	\begin{equation*}
		\gamma (y,z,x)=\gamma (z,y,x)		
		=\gamma (y,x,z).                                         
	\end{equation*} 
	Therefore,  $\gamma\in S^3(  J,\R)$.\\
	
	Conversely, we assume that $\gamma\in S^3({ J},\R)$ and $  d_{r}^3\gamma =0 $.\\
	First, we  verify   the symmetric condition for $[\cdot,\cdot]$:\\
	By $ \gamma\in S^3(J,\R) $, we have	$\gamma(x,y,z)=\gamma(y,x,z)$, Hence  $B([x,y],z)=B([y,x],z)$. 	
	Since $B$ is nondegenerate one can deduce $[x,y]=[y,x]$.\\
	Next, we verify the Eq.\eqref{invDef}. For any $x,y,z\in J$,we have
	$ \gamma(x,y,z)=\gamma(y,z,x) $, that is $ B([x,y],z)=B([y,z],x)  $. Then $B([x,y],z)=B(x,[y,z])$.
	So, Eq.\eqref{invDef} holds.\\
	Now, we prove the Hom-Jacobi-Jordan identity.\\
	For all $x,y,z\in J$, by  
	Eq.\eqref{invDef}, we have
	\begin{align*}
		\eqref{invariant2}=	B\big(\left[ [x,y],\alpha(z) \right],t    \big)
		+B\big( \left[[y,z],\alpha(x) \right],t    \big)
		+B\big( \left[[x,z],\alpha(y) \right],t    \big).   
	\end{align*}
	Thus, \[d_r^3\gamma(x,y,z,t)=2\Big(B\big(\left[ [x,y],\alpha(z) \right],t    \big)
	+B\big( \left[[y,z],\alpha(x) \right],t    \big)
	+B\big( \left[[x,z],\alpha(y) \right],t    \big)\Big).  \]
	Since $d_r^3\gamma=0$ and $B$ is nondegenerate, we get  the Hom-Jacobi identity.\\
	Finally, we prove the coadjoint representation identity.\\
	Since \eqref{invDef} and \eqref{alphaInvariant} are satisfied, we have
	$d_r^3\gamma(x,y,z,t)=\eqref{Oinvariant1}+\eqref{Oinvariant2}$.
	Since $ d_r^3\gamma(x,y,z,t)=0 $ and $ \eqref{Oinvariant1}=0 $, we obtain $\eqref{Oinvariant2}=0  $. This finishes the proof.
\end{proof}
\section{Second cohomology group 
	of metric Hom-Jacobi-Jordan-algebra}  
The task of this section is to introduce the second cohomology group of metric Hom-Jacobi-Jordan-algebra, which we will use to describe the quadratic
extensions.

\subsection{Construction of a $2$-Coboundary operators for metric  Hom-Jacobi-Jordan  algebra }	
Let $M=J\oplus \mathfrak a  $ be a Hom-Jacobi-Jordan algebra with  structure $ \alpha_{M}=\alpha+\beta  $ where $\alpha\colon J\to J$, $ \beta\colon \mathfrak a\to  \mathfrak a$ and 
$[\cdot,\cdot]_{M}$  such that $\mathfrak a $ is an abelian ideal of  $M$. Then by Theorem \ref{structureMAY24}, $[\cdot,\cdot]_{M}=[\cdot,\cdot]+\rho+\theta $, where $ (J,[\cdot,\cdot],\alpha) $ is a Hom-Jacobi-Jordan algebra,  $ \rho  $ is a representation of $J$ on $\mathfrak a$ and $\theta$ is a $2$-cocycle of $J$ on $\mathfrak a$. 
Let $ \mathfrak n=M\oplus J^* $,  $d'\colon \mathfrak n^2\to \mathfrak n $ a bilinear symmetric map
and $\widetilde{\alpha}\colon  \mathfrak n \to  \mathfrak n $
a linear map given by
$ \alpha_{\mathfrak n}(x+v+Z)=\alpha_{M}(x+v)+\alpha^*(Z)) $
for all $x\in J$, $v\in V$, $Z\in J^*$.\\			
Let 
$$0\longrightarrow (J^*,\alpha^*)\stackrel{i}{\longrightarrow} ( \mathfrak n,[\cdot,\cdot]_{\mathfrak n},\alpha_{\mathfrak n})\stackrel{\pi}{\longrightarrow }(M,[\cdot,\cdot]_{M},\alpha_{M}) \longrightarrow 0,$$
be an abelian extension of $ M $ by $ J^* $. Then (by Theorem \ref{structureMAY24}) $[\cdot,\cdot]_{\mathfrak n}=[\cdot,\cdot]_{M}+ \rho^{*}+\gamma'$
where  $ \rho^{*}  $ is a representation of $M$ on $J^*$ and $\gamma'$ is a $2$-cocycle of $M$ on $J^*$. 
Hence,		
\begin{align}
	[x,y]_{\mathfrak n} &=[x,y]+\theta(x,y)+\gamma'(x,y);\label{1structureMetric}\\
	[x,v]_{\mathfrak n}&=\rho(x)v+\gamma'(x,v);\label{2structureMetric}\\
	[v,w]_{\mathfrak n}&=\gamma'(v,w);\label{3structureMetric}\\
	[Z,x]_{\mathfrak n}&=	\rho^{*}(x)Z\\
	[Z,v]_{\mathfrak n}&=\rho^{*}(v)Z \\
	[Z_1,Z_2]_{\mathfrak n}&=0      
\end{align}
for all $x\in J, \, v\in V, Z_1,Z_2 \in J^*$.\\
Let  $B\colon \mathfrak{n}^2\to \R  $ be a  bilinear form such that $ \left(\mathfrak n,[\cdot,\cdot]_{\mathfrak n},   \alpha_{\mathfrak n}, B \right)  $ is a 	metric  Hom-Jacobi-Jordan algebra, the ideals  $J$ and $ J^* $ are isotropic and
\begin{equation}\label{BJ*}
	B(Z,x+v)=Z(x)
\end{equation} 
for all  $Z\in J^*,\, x\in J,\, v\in \mathfrak{a}$.
\begin{lem}
	Under the above notation, 
	we have 
	\begin{equation*}
		[Z,x]_{\mathfrak n}=Z([x,\cdot])  \text{ and }  [Z,v]_{\mathfrak n}=0   
	\end{equation*}
	for all $Z\in J^*,\, x\in J, v\in \mathfrak a$.
\end{lem}
\begin{proof}
	Let $Z\in J^*,\, x\in J, v\in \mathfrak a$. We have $B\left(Z,v  \right)=Z(v)=0$. Then, 
	$B\left(Z,[x,y]_{\mathfrak n}  \right)=Z([x,y])  $.
	Moreover, by invariance of $B$ we have $B\left(Z,[x,y]_{\mathfrak n}  \right)=B\left([Z,x]_{\mathfrak n},y  \right)$. Hence, $ \rho^*(x)Z(y)=Z([x,y])$. Which implies that $[Z,x]_{\mathfrak n}=Z([x,\cdot]) $.\\
	Now, we show that $ [Z,v]_{\mathfrak n}=0 $. Since $ J^* $	 is an ideal of $\mathfrak n  $, according to  Proposition  \ref{IdealorthogonalCentre} we have $(J^{*})^{\perp}\subset \mathfrak Z(J^{*})  $. Then, $ \mathfrak a \subset \mathfrak Z(J^{*})$ since $B\left(Z,v  \right)=0$. Therefore $[Z,v]_{\mathfrak n}=0$.
\end{proof}
\begin{prop}\label{PropSep17}
	For all $v,w\in \mathfrak a$, we have
	\begin{equation}
		B_{\mathfrak a}\left(\beta(v),w \right) =B_{\mathfrak a}\left(v,\beta(w)\right) .      
	\end{equation}	                  
\end{prop}
\begin{proof}
	By eq.\eqref{alphadef}, we have  $B\left((\alpha+\beta+\alpha')(v),w \right)=B\left(v ,(\alpha+\beta+\alpha')(w)\right)$. Therefore, $ B_{\mathfrak a}\left(\beta(v),w \right)=B_{\mathfrak a}\left(v,\beta(w)\right)  $.
\end{proof}
\begin{thm}
	\label{1Sep23}
	If $ \left(\mathfrak n,[\cdot,\cdot]_{\mathfrak n},\alpha_{\mathfrak n} \right)  $ is a 	  metric Hom-Jacobi-Jordan algebra. Then, for all $x,y\in J$, $v,w\in \mathfrak a$, $ Z\in J^* $, we have
	\begin{equation}\label{1d}
		\begin{aligned}
			[x,y]_{\mathfrak n}&=[x,y]+\theta(x,y)+\gamma(x,y,\cdot);\\
			[x,v]_{\mathfrak n}&=\rho(x)v
			+B_{\mathfrak a}\left(\theta(\cdot,x),v \right) ;\\
			[v,w]_{\mathfrak n}&=B_{\mathfrak a}\left(\rho(\cdot)v,w \right);\\
			[Z,x]_{\mathfrak n}&=Z\left([x,\cdot] \right);\\
			[Z_1,v+Z_2]_{\mathfrak n}&=0.      
		\end{aligned}
	\end{equation}	
	where $\gamma\in S^3(J,\R)  $.		  
	\label{lemmaGamma}	                     
\end{thm}
\begin{proof} Assume that $ \left(\mathfrak n,[\cdot,\cdot]_{\mathfrak n},\alpha_{\mathfrak n} \right)  $ is a 	  metric Hom-Jacobi-Jordan algebra.
	Let $ \gamma(x,y,z)=\gamma'(x,y)(z) $.
	By eq.\eqref{invDef}, we have $B\left([x,y]_{\mathfrak n},z \right)=B\left( x, [y,z]_{\mathfrak n}\right)   $. Thus, using \eqref{1structureMetric}, we have  
	$\gamma'(x,y)(z)=\gamma'(y,z)(x).$	       
	Hence $ \gamma(x,y,z)=\gamma(y,z,x) $. Moreover, by $[x,y]_{\mathfrak n}=[y,x]_{\mathfrak n}$, hence 	
	$ \gamma(x,y,z)=\gamma(y,x,z) $. By repeating this process, we obtain that $ \gamma\in S^3(J,\R). $\\	
	Now we prove that $ \gamma'(x,v)(y)= B_{\mathfrak a}\left(\theta(y,x),v \right) $.\\ By eq.\eqref{invDef}, we have $B\left([y,x]_{\mathfrak n},v \right)=B\left( y, [x,v]_{\mathfrak n}\right)   $.	Thus, using \eqref{1structureMetric}, \eqref{2structureMetric}  and \eqref{BJ*}, we have $ \gamma'(x,v)(y)= B_{\mathfrak a}\left(\theta(y,x),v \right). $\\
	For $ \gamma'(v,w) $, by  \eqref{invDef}, we have $B\left([x,v]_{\mathfrak n},w \right)=B\left( x, [v,w]_{\mathfrak n}\right)   $. Thus, using
	\eqref{2structureMetric}, \eqref{3structureMetric} and \eqref{BJ*}, we  have
	$\gamma'(v,w)(x)=	B_{\mathfrak a}\left(\rho(x)v,w \right)$	.  
	Hence 
	\begin{gather}\label{gammaprime2}
		\gamma'(v,w)=	B_{\mathfrak a}\left(\rho(\cdot)v,w \right).
	\end{gather}
	The proof is completed.
\end{proof}
\begin{defn}
	A Quadratic representation of a Hom-Jacobi-Jordan algebra 
	$ \left(J,[\cdot,\cdot],\alpha \right)  $ on a vector space $\mathfrak a$
	with respect to $ \beta\in End(\mathfrak a)$   consists of a triple $ \left(\rho,\mathfrak a,B_{\mathfrak a} \right)  $, where $ \rho \colon J\to End(\mathfrak a) $ is a representation of the 
	Hom-Jacobi-Jordan algebra $ J $ on $ \mathfrak a $ with respect to $ \beta\in End(\mathfrak a)$, and
	$ B_{\mathfrak a}\colon \mathfrak a\times \mathfrak a\to \R $ a symmetric bilinear form, satisfying, 
	\begin{align}
		B_{\mathfrak{a}}\left(  \rho(x)(v),w\right)& =B_{\mathfrak{a}}\left( v,\rho(x)(w)\right)\label{EqQrepr1}
	\end{align}	
	for all $x,y\in  J$  and $v,w\in  \mathfrak{a}$.
\end{defn}
\begin{lem}\label{1LemmaqSep18}
	If $ \left(\mathfrak n,[\cdot,\cdot]_{\mathfrak n},\alpha_{\mathfrak n} ,B\right)  $ is a 	  metric Hom-Jacobi-Jordan algebra, $ \left(\rho,\mathfrak a,B_{\mathfrak a} \right)  $ is quadratic representation of $J$ on $\mathfrak a$.  
\end{lem}
\begin{proof}
	By Lemma \ref{lemmaGamma}, we have $\gamma'(v,w)=	B_{\mathfrak a}\left(\rho(\cdot)v,w \right)  $. Using, in addition, the symmetry of the bracket $  [\cdot,\cdot]_{\mathfrak n}$, 
	we obtain 	$B_{\mathfrak a}\left(\rho(\cdot)v,w \right)=B_{\mathfrak a}\left(\rho(\cdot)w,v \right)$. This finishes the proof.	              
\end{proof}
Define a bilinear multiplication on $S^p(  J,\mathfrak  a)\times S^q(  J,\mathfrak  a)$ by
\begin{align}\label{Wedge}
	B_{\mathfrak a}(f\wedge g)(x_1,\cdots,x_{p+q})=\sum_{\sigma\in Sh(p,q)}
	B_{\mathfrak a}\left( f (x_{\sigma(1)},\cdots,x_{\sigma(p)}),g (x_{\sigma(p+1)},\cdots,x_{\sigma(p+q)})\right)                                                 
\end{align} 
where 
$ Sh(p,q) $ are the permutations
in $ \mathfrak{S}_{p+q} $ which are increasing on the first $p$ and the last $q$ elements.
\begin{prop}\label{1LemmaSep17}
	Let $ \left(\mathfrak n, [\cdot,\cdot]_{\mathfrak n},\alpha_{\mathfrak n},B \right)  $ be a 	  metric Hom-Jacobi-Jordan algebra. For $f, g\in C^2_{\alpha,\beta}(J,\mathfrak a)$, we have 
	\begin{equation*}
		B_{\mathfrak a}\left(f(\alpha(x),\alpha(y) ),g(z,t )  \right) = B_{\mathfrak a}\left(f(x,y ),g(\alpha(z),\alpha(t) )  \right)	                   
	\end{equation*}   	                                            
\end{prop}
\begin{proof}
	Since $f,g\in C^2_{\alpha,\beta}(J,\mathfrak a)$, we have, $ f\circ \alpha=\beta \circ f  $ and  $ g\circ \alpha=\beta \circ g   $.\\ According to  Proposition \ref{PropSep17}, we have 
	$B_{\mathfrak a}\left(\beta\circ f(x,y), g(x,z)\right)
	=B_{\mathfrak a}\left( f(x,y), \beta\circ g(x,z)\right)  $. 
	It follows that
	$B_{\mathfrak a}\left(f(\alpha(x),\alpha(y)), g(z,t)\right)
	=B_{\mathfrak a}\left( f(x,y), g(\alpha(z),\alpha(t))\right)$.	    
\end{proof}
\begin{prop}\label{Lemmaquadratic2cocycle}
	If $ \left(\mathfrak n,[\cdot,\cdot]_{\mathfrak n},\alpha_{\mathfrak n},B \right)  $ is a 	  metric Hom-Jacobi-Jordan algebra,	then
	the pair  $(\theta,\gamma) $ satisfies the following properties
	\begin{align*}
		d^2\theta(x,y,z)&=0\\
		d_{r}^{3}\gamma(x,y,z,\alpha(a)) 
		+\frac{1}{2}B_{\mathfrak a}\left(\theta\wedge ( \theta\circ \alpha)\right)(x,y,z,a)&=0	            
	\end{align*}		
	for all $x,y,z,a\in J$.		     
\end{prop}
\begin{proof}
	We have, $ \left(M,[\cdot,\cdot]_{M},\alpha_{M} \right)  $ is a Hom-Jacobi-Jordan algebra, $\left( J^*,\rho^{*},\alpha^{*}\right) $ is a representation of the Hom-Jacobi-Jordan  algebra $ M $, 
	$ \mathfrak n=M\oplus J^* $ and $[\cdot,\cdot]_{\mathfrak n}=[\cdot,\cdot]_{M}+\gamma'$.
	By 	Theorem \ref{structureMAY24}, it follows    that $d^2\gamma'=0$. For all $x,y,z,a\in J$, we have
	\begin{align*}
		d^2\gamma'(x,y,z)(t) &=\gamma'([x,y]_{M},\alpha_{M}(z))(t)+\gamma'([x,z]_{M},\alpha_{M}(y))(t)+\gamma'([y,z]_{M},\alpha_{M}(x))(t)\\
		+&\rho'(\alpha_{M}(z))\gamma'(x,y)(t)+\rho'(\alpha_{M}(x))\gamma'(y,z)(t)+\rho'(\alpha_{M}(y))\gamma'(x,z)(t),
	\end{align*}
	where $ t=\alpha(a) $. Then, by $ [x,y]_{M}=[x,y]+\theta(x,y) $, $ \gamma'(x,v)(y)= B_{\mathfrak a}\left(\theta(y,x),v \right) $ and $ \gamma'(v,w)=	B_{\mathfrak a}\left(\rho(\cdot)v,w \right) $, we obtain
	\begin{align}
		d^2\gamma'(x,y,z)(t) &	=\gamma([x,y],\alpha(z),t)
		+	\gamma([x,z],\alpha(y),t)
		+	\gamma([y,z],\alpha(x),t)\label{Asep17}\\
		+&\gamma\left( x,y,[\alpha(z),t] \right) 
		+\gamma\left( y,z,[\alpha(x),t] \right) 
		+\gamma\left( x,z,[\alpha(y),t] \right)\label{Bsep17} \\
		+&B_{\mathfrak a}\left(\theta(\alpha(a),\alpha(z)), \theta(x,y)\right) 
		+B_{\mathfrak a}\left(\theta(\alpha(a),\alpha(y)), \theta(x,z)\right) \nonumber\\
		+&B_{\mathfrak a}\left(\theta(\alpha(a),\alpha(x)), \theta(y,z)\right).\nonumber
	\end{align}
	Using $\eqref{Asep17}+\eqref{Bsep17}=d_{r}^{3}\gamma(x,y,z,t)$ and Proposition \ref{1LemmaSep17}, one gets
	\begin{gather*}
		d^2\gamma'(x,y,z)(t)	=d_{r}^{3}\gamma(x,y,z,t) 
		+\frac{1}{2}B_{\mathfrak a}\left(\theta\wedge ( \theta\circ \alpha)\right)(x,y,z,a). 
	\end{gather*}
	Hence, $d_{r}^{3}\gamma(x,y,z,\alpha(a)) 
	+\frac{1}{2}B_{\mathfrak a}\left(\theta\wedge ( \theta\circ \alpha)\right)(x,y,z,a)=0.  $             
\end{proof}
Bringing these results, we provide the following definitions.
\begin{defn}The pair
	$(\theta,\gamma) $ is called a  quadratic  $2$-cochain if $ \theta\in C_{\alpha,\beta}^2(J,\mathfrak a) $ and $\gamma\in C^3_{r}(J,\R)$.	
	Denote $ C_{Q}^2(  J,\mathfrak  a) $
	the set of quadratic  $2$-cochains.\\ 
	We define a  map by $ d_{Q}^{2}\colon C_{Q}^2( J,\mathfrak  a)\to C^3_{r}(  J,\mathfrak  a)\times C^4(  J,\R) $ by 
	\begin{gather}
		d^2_{Q}\left( \theta,\gamma\right)(x,y,z)(t)\nonumber\\ =\left(d^2\theta(x,y,z),d^3_{r}\gamma(x,y,z,t)
		+\frac{1}{2}B_{\mathfrak a}\left(\theta\wedge(\theta \circ \alpha ) \right) (x,y,z,a)\right)\label{2Qcobrd}                   
	\end{gather}  
	where $ t=\alpha(a) $.\\				
	$(\theta,\gamma) $ is called a  quadratic $2$-Hom-cocycle of $  J $ on $ \mathfrak a $ if and only if 
	$d^2_{Q}(\theta,\gamma)=0 $.
	We denote  $   Z^2_{Q}( J,\mathfrak a) $ the group generated by any  quadratic $2$-cocycle on $ \mathfrak a $.	          
\end{defn}

\subsection{Construction of a $1$-Coboundary operators for metric  Hom-Jacobi-Jordan  algebra}
We aim in this section to construct a linear map $d^1_Q$ satisfying $d^2_Q\circ d^1_Q=0$ and then the second cohomology group 
of metric Hom-Jacobi-Jordan-algebra.
\begin{prop}\label{17OCTpropWedge}
	Let $ f\in C_{\alpha,\beta}^2(J,\mathfrak a) $ and $ g \in C_{\alpha,\beta}^1(J,\mathfrak a) $.  We have	
	\begin{align*}
		&d^{3}_{r} B_{\mathfrak a}(f\wedge g)(x,y,z,t)=B\left( d^2f(x,y,z),g(t)\right) \\
		+&	B\left(d^2_{c} f(x,y,t),g(z) \right)	
		+B\left(d^2_{c} f(x,z,t),g(y)  \right)
		+B\left(d^2_{c} f(y,z,t),g(x)  \right)\\
		+&B\left(  (f\circ \alpha)\wedge d^1g        \right)(x,y,z,a)	
	\end{align*}
	for any $x,y, z,a\in J$ and $t=\alpha(a) $.
\end{prop} 
\begin{proof}
	Let $ f\in C_{\alpha,\beta}^2(J,\mathfrak a) $ and $ g \in C_{\alpha,\beta}^1(J,\mathfrak a) $.
	We take 	    
	$\gamma=B_{\mathfrak a}\left(f\wedge g\right) $.\\
	For any $x,y, z,a\in J$ 
	and $t=\alpha(a) $,	
	we have
	\begin{align}
		&d^{3}_{r}\gamma(x,y,z,t)\nonumber\\
		&=		\gamma([x,y],\alpha(z),t)
		+\gamma([x,z],\alpha(y),t)
		+\gamma([y,z],\alpha(x),t)\nonumber\\		
		+& \gamma\left( x,y,[\alpha(z),t] \right)
		+\gamma\left( y,z,[\alpha(x),t] \right)
		+\gamma\left( x,z,[\alpha(y),t] \right)\nonumber\\
		&=\circlearrowleft_{x,y,z}\big( 	\gamma([x,y],\alpha(z),t)
		+\circlearrowleft_{x,y,z}  \gamma\left( x,y,[\alpha(z),t] \right)\big) \nonumber\\
		&=\circlearrowleft_{x,y,z}\big( B(f([x,y],\alpha(z)),g(t))
		+B(f([x,y],t),g(\alpha(z)))big) 
		+B(f(\alpha(a),\alpha(z)),g([x,y]))\label{1cohomologyOct14}\\
		+&\circlearrowleft_{x,y,z}\big( B(f(x,y),g([\alpha(z),t]))
		+B\left(f(x,[\alpha(z),t]),g(y)
		+B\left(f(y,[\alpha(z),t]),g(x)\right)\right)\big) \label{2cohomologyOct14}		
	\end{align}
	where $\circlearrowleft_{x,y,z}  $ denotes summation over the cyclic permutation on $ x,y,z $.\\
	By Proposition \ref{PropSep17} and $ g\in C_{\alpha,\beta}^{1}(J,\mathfrak a) $, we have
	\begin{equation*}
		\circlearrowleft_{x,y,z}	B(f([x,y],t),g(\alpha(z)))=\circlearrowleft_{x,y,z}B\left(\beta( f([x,y],t)),g(z)\right) 
	\end{equation*}
	Hence, 
	\begin{align*}
		\eqref{1cohomologyOct14}
		&=B\Big(d^2f(x,y,z),g(t)\Big)
		-\circlearrowleft_{x,y,z}B\Big(\rho\left( \alpha(x)\right) f(y,z)
		,g(t)\Big)\\
		+&\circlearrowleft_{x,y,z}B\left(\beta( f([x,y],t)),g(z)\right) 
		+\circlearrowleft_{x,y,z}B(f(\alpha(a),\alpha(z)),g([x,y])). 
	\end{align*}		
	For \eqref{2cohomologyOct14}, we have
	\begin{gather*}
		\circlearrowleft_{x,y,z}	B\left( f(x,y),g([\alpha(z),t])\right) 
		=\circlearrowleft_{x,y,z}B\left( f(\alpha(x)), \alpha(y),g([z,a]))\right) .
	\end{gather*}	
	Then,
	\begin{gather*}
		\eqref{1cohomologyOct14}+\eqref{2cohomologyOct14}
		=B\Big(d^2f(x,y,z),g(t)\Big)
		-\circlearrowleft_{x,y,z}B\Big(\rho\left( \alpha(x)\right) f(y,z)
		,g(t)\Big)\\
		+\circlearrowleft_{x,y,z}B\left(\beta( f([x,y],t)),g(z)\right)
		+B\left(f(x,[\alpha(z),t]),g(y)\right)
		+B\left(f(y,[\alpha(z),t]),g(x)\right)\\ 
		+\circlearrowleft_{x,y,z}B(f(\alpha(a),\alpha(z)),g([x,y]))
		+\circlearrowleft_{x,y,z}B(f(\alpha(x), \alpha(y),g([z,a])) .
	\end{gather*}		
	On the other hand, we have
	\begin{align*}
		&\beta \left( f([x,y],t)\right) 			 
		+f(y,[\alpha(x),t])		 
		+f(x,[\alpha(y),t])\\
		&=d^2_{c} f(y,z,t)
		-\rho(y)f( \alpha(z),t  )	
		-\rho(y)f( \alpha(z),t  )
		-\beta \left( \rho(t)f(y,z)     \right),	              
	\end{align*}
	and
	\begin{align*}
		B_{\mathfrak a}\left( \rho(y)f( \alpha(z),t  ),g(x)\right)
		&=B_{\mathfrak a}\left( f( \alpha(z),t  ),\rho(y)g(x)\right) \\
		&=B_{\mathfrak a}\left( f\circ \alpha( z,a ),\rho(y)g(x)\right).
	\end{align*}
	Moreover, we have
	\begin{align*}
		B_{\mathfrak a}\left(\beta \left( \rho(t)f(y,z)     \right)	,g(x)\right)
		&=B_{\mathfrak a}\left( \rho(\alpha(a))f(y,z)    	,\beta \left(g(x) \right)\right)\\
		&=B_{\mathfrak a}\left( f(y,z)    	,\rho(\alpha(a))\beta \left(g(x) \right)\right)\\
		&=B_{\mathfrak a}\left( f(y,z)    	,\beta\left( \rho(a)g(x) \right)\right)\\
		&=B_{\mathfrak a}\left(\beta\left( f(y,z)    \right)	, \rho(a)g(x) \right)\\
		&=B_{\mathfrak a}\left( f\circ \alpha(y,z)    	, \rho(a)g(x) \right).
	\end{align*}	
	Therefore, by 
	straightforward computations, we obtain
	\begin{align*}
		d^{3}_{r}\gamma(x,y,z,t)	&=B(d^2f(x,y,z)),g(t))\\
		+&B(d^2_{c}f(x,y,t)),g(z))+B(d^2_{c}f(x,z,t)),g(y))
		+B(d^2_{c}f(y,z,t)),g(x))
		\\
		+&B((f\circ \alpha) \wedge d^{1}_{c} g  )(x,y,z,a).
	\end{align*}	        
\end{proof}	
\begin{rem}
	If $\alpha=id_J$ and $ \beta=Id_{\mathfrak a} $, we have
	\begin{gather*}
		d^{3}_{r}\gamma	=B(d^2f\wedge g)+B\left(f\wedge d^1g \right) 
	\end{gather*}	                     
\end{rem}
\begin{lem}
	Let $ (\theta,\gamma) $ and  $ (\theta',\gamma') $  two quadratic $2$-cochains.
	\begin{gather}
		d^2_{Q}	(\theta,\gamma)=d^2_{Q}	(\theta',\gamma')	\iff  \left\{ \begin{array}{ll}
			\theta'=\theta+d^1\tau \\
			d_{r}^3\gamma'=	d_{r}^3\gamma  -\frac{1}{2}d_{r}^3	B(\tau\wedge d^1\tau )
			-d_{r}^3 B( \tau  \wedge  \theta  )		
			\label{Jun13}	\\
		\end{array} \right.
	\end{gather}
\end{lem}
\begin{proof}
	Let $ (\theta,\gamma) $ and  $ (\theta',\gamma') $  two quadratic $2$-cochain such that $ d^2_{Q}	(\theta,\gamma)=d^2_{Q}	(\theta',\gamma') $. Then,
	\begin{gather}
		d^2\theta=d^2\theta'\label{conditionJune13}
	\end{gather}
	and
	\begin{gather}
		d_{r}^3\gamma+\frac{1}{2}B(\theta\wedge(\theta\circ  \alpha)) \label{1conditionJune13}  =d_{r}^3\gamma'+\frac{1}{2}B(\theta'\wedge(\theta'\circ  \alpha )). 
	\end{gather}
	Eq.\eqref{conditionJune13} implies that there exist a  $1$-Hom-cohain $\tau$ satisfies 
	\begin{gather}\label{tau}
		\theta'=\theta+d^1\tau	
	\end{gather}
	Thus, using \eqref{1conditionJune13}, we have
	\begin{align}
		d^3_{r}\gamma &=d^3_{r}\gamma'+\frac{1}{2}B\Big(\left( \theta+d^1\tau\right) \wedge(\left( \theta+d^1\tau\right) \circ  \alpha )\Big) - \frac{1}{2}B(\theta\wedge(\theta\circ  \alpha))\nonumber\\
		&=d^3\gamma' 
		+\frac{1}{2}B\Big( \theta \wedge(d^1\tau \circ  \alpha )\Big) 
		+\frac{1}{2}B\Big( d^1\tau \wedge( \theta\circ  \alpha )\Big) 
		+\frac{1}{2}B\Big(d^1\tau  \wedge(d^1\tau\circ  \alpha )\Big).\label{this}
	\end{align}
	Hence, by Proposition \ref{1LemmaSep17}, we obtain $B\left(  \theta \wedge(d^1\tau \circ  \alpha )\right) =B\left(  d^1\tau \wedge( \theta\circ  \alpha )\right)   $.
	Therefore,
	\begin{equation*}	
		d^3\gamma 
		=d^3\gamma'
		+B( d^1\tau  \wedge  (\theta\circ  \alpha )  )	
		+\frac{1}{2}B( d^1\tau  \wedge   (d^1\tau\circ  \alpha )).
	\end{equation*}
	Replacing  $f$, $g$ by $ d^1\tau $,  $ \tau $, in Proposition \ref{17OCTpropWedge} and since
	By $ d^2\circ  d^1(\tau)=0 $, 
	we have
	\begin{equation}\label{d3B17oct}
		d^{3}_{r}B(d^1_{c}\tau\wedge \tau)(x,y,z,t)	=B((d^1_{c}\tau\circ \alpha) \wedge d^{1}_{c} \tau  )(x,y,z,a).
	\end{equation}
	Replacing  $f$, $g$ by $ \theta $, $ \tau $,
	in Proposition \ref{17OCTpropWedge}, by that fact 
	$ d^2\theta=0 $ and that $ d^2_{c}\theta=0 $, we have
	\begin{equation*}
		d^{3}_{r}B(\theta\wedge \tau)(x,y,z,t)	=B((\theta\circ \alpha) \wedge d^{1}_{c} \tau  )(x,y,z,a).
	\end{equation*}	
	Therefore,
	\begin{equation*}
		d^3\gamma 
		=d^3\gamma'+d^{3}_{r}(\theta\wedge \tau)+\frac{1}{2}
		d^{3}_{r}(d^1_{c}\tau\wedge \tau)
	\end{equation*}
	Hence,
	\begin{gather}
		d^3\gamma'=	d^3\gamma  -\frac{1}{2}d^3	B(\tau\wedge d^1\tau )
		-d^3B\left(\theta\wedge \tau \right)	.
	\end{gather}	                         
\end{proof}	  
\begin{thm}\label{17octequiv}
	If $d_{c}^2\theta=0$, there exists $ \sigma \in C^3(J,\R) $ such that
	\begin{gather*}
		d^2_{Q}	(\theta,\gamma)=d^2_{Q}	(\theta',\gamma')
		\iff  \left\{ \begin{array}{ll}
			\theta'=\theta+d^1\tau \\
			\gamma'=	\gamma+d^2\sigma  -	B\left( \tau\wedge (\theta+\frac{1}{2}d^1\tau )\right) 	
			\\
		\end{array} \right.
	\end{gather*}			                    
\end{thm}						
Using the previous observations, we give the following definitions. 		
\begin{defn}	
	Define a linear map $ d^1_{Q}\colon C^1_{Q}(\mathfrak{J},\mathfrak a)\to C^2_{Q}(\mathfrak{J},\mathfrak a) $  by 
	\begin{gather*}
		d^1_{Q}(\tau,\sigma)=\big(d^1\tau,d_{r}^2\sigma-\frac{1}{2}	B\left( \tau\wedge  d^1\tau \right)\big) .                  
	\end{gather*}	 	 						
	A  quadradic $ 2 $-cochain  $(\theta, \gamma) $  is called quadradic  $2$-cobord  if and only if there exist a quadratic $1$-cochain $(\tau,\sigma)$ satisfies 
	$  d^1_{Q}(\tau,\sigma)=(\theta, \gamma) $. Denote $ B^{2}_{Q}\left(\mathfrak{J},\mathfrak a \right)  $ the groupe generated by the set of all  quadratic $2$-cobords.                                  
\end{defn}
\begin{prop}
	Any quadratic $2$-cobord is a quadratic $2$-cocycle (i.e $d^2_{Q} \circ d^1_{Q} =0$).                        
\end{prop}
\begin{proof}
	We set  $\theta=d^1\tau$ and $\gamma=d^2\sigma-\frac{1}{2}B_{\mathfrak a}(d^1\tau\wedge\tau) $. Using  \eqref{d3B17oct},  we have
	$d^3\gamma=
	-\frac{1}{2}B_{\mathfrak a}\left( d^1\tau\wedge  (d^1\tau\circ  \alpha) \right)   $. 
		Hence, by \eqref{2Qcobrd}
		\begin{align*}
			d^2_{Q}(\theta,\gamma)&=(d^2\theta,d_{r}^3\circ d_{r}^2\sigma -\frac{1}{2}B_{\mathfrak a}\left( d^1\tau\wedge (d^1\tau\circ  \alpha)\right)  + \frac{1}{2}B\left(d^1\tau\wedge (d^1\tau\circ  \alpha) \right))\\
			&=(0,0) .	
		\end{align*}				
	\end{proof}
	\begin{defn}
		The $2^{th}$ quadratic  cohomology group of the metric Hom-Jacobi-Jordan  algebra $\mathfrak  J $  on $ \mathfrak a\times  \mathfrak  J^*$,
		the quotient 
		\[ H^2_{Q}(\mathfrak  J,\mathfrak a )=Z^2_{Q}(\mathfrak  J,\mathfrak a )/B^2_{Q}(\mathfrak  J,\mathfrak a ).\]	
	\end{defn}
	\section{Quadratic extensions}	
	In this section, we study quadratic extensions of Hom-Jacobi-Jordan algebras and  we show that they are classified by the cohomology group $H^2_{Q}(\mathfrak  J,\mathfrak a )$.	
	Let $ \left( \mathfrak J ,[\cdot,\cdot]_{\mathfrak J},\alpha_{\mathfrak J} ,B\right) $ be a metric of Hom-Jacobi-Jordan algebra and $I$ an  isotropic ideal of $\mathfrak J  $. For all $x,y\in \mathfrak J$, we denote $[\pi_n(x),\pi_n(y)]_{\overline{\mathfrak J}}=\pi_n\left([x,y] \right)    $, $ \overline{\alpha_{\mathfrak J} }(\pi_n(x))=\pi_n\circ \alpha_{\mathfrak J}(x) $ and $\overline{B}(\pi_n(x),\pi_n(y))=B(x,y)  $ where  $ \pi_n $
	is the natural projection  $\mathfrak J\to \mathfrak J/I $. If $ i\colon \mathfrak{a}\to \mathfrak J $ is a homomorphism, we denote $\overline{i}=\pi_n \circ i  $. 
	\begin{defn}
		Let $ \left(J,[\cdot,\cdot],\alpha \right) $ be a Hom-Jacobi-Jordan algebra, $I$ an isotropic ideal  in $  J $
		and $ \left(\mathfrak{a},\rho,\beta,B_{\mathfrak{a}} \right)  $ a quadratic representation
		of $J$.
		A quadratic extension $ \left(\mathfrak J,I,i,\pi \right)  $ of  $ J $ by $ \mathfrak{a} $ is an exact sequence
		$$0\longrightarrow \left(\mathfrak{a} ,\rho,\beta\right)  \stackrel{\overline{i}}{\longrightarrow}\left( \mathfrak J/I ,[\cdot,\cdot]_{\overline{\mathfrak J}}  ,\overline{\alpha_{\mathfrak J} },\overline{B}\right)  \stackrel{\pi}{\longrightarrow } \left( J,[\cdot,\cdot],\alpha\right)  \longrightarrow 0$$
		satisfying $ \left( \mathfrak J ,d,\alpha_{\mathfrak J },B\right) $ is a metric Hom-Jacobi-Jordan algebra, $\overline{\alpha_{\mathfrak J} }\circ \overline{i}=\overline{i}\circ \beta  $, $ \alpha\circ \pi=\pi \circ \overline{\alpha_{\mathfrak J} } $, 
		$\overline{i} (\mathfrak{a})=I^{\perp}/I$	 and that $ \overline{i}\colon \mathfrak{a}\to I^{\perp}/I $ is an isometry.						
	\end{defn}
	\begin{prop}
		Let
		\begin{equation}\label{qEXTENSION}
			0\longrightarrow \mathfrak{a}\stackrel{i}{\longrightarrow} \mathfrak J/I\stackrel{\pi}{\longrightarrow } J \longrightarrow 0,                            
		\end{equation} 		
		be an extension of 	$J$ by  $\mathfrak{a}$ such that $i\colon \mathfrak{a} \to i\left( \mathfrak{a}\right)  $ is an isometry. Then,
		the quadruple   $ \left(\mathfrak J,I,i,\pi \right)  $ define a quadratic extension if and only if the following sequence define  an extension of  $ \mathfrak{J}/I $	by $J^*$:
		\begin{equation}\label{qEXTENSION2}
			0\longrightarrow J^*\stackrel{\tilde{\pi}^*}{\longrightarrow} \mathfrak J\stackrel{\pi_n}{\longrightarrow } \mathfrak J/I \longrightarrow 0,
		\end{equation}	
		where $ \pi_n $
		is the natural projection  $\mathfrak J\to \mathfrak J/I $,  
		$\tilde{\pi}=\pi\circ \pi_n  $, $ \tilde{\pi}^* $ the dual map of $ \tilde{\pi} $
		where we identify $ J^* $ with $J$.
	\end{prop}
	\begin{proof}
		We have
		\begin{equation*}
			0\longrightarrow \mathfrak{a}\stackrel{i}{\longrightarrow} \mathfrak J/I\stackrel{\pi}{\longrightarrow } J \longrightarrow 0,                            
		\end{equation*} 		
		and extension of 	$J$ by  $\mathfrak{a}$ such that $i\colon \mathfrak{a} \to i\left( \mathfrak{a}\right)  $ is an isometry.
		Then, 
		\begin{align}
			\overline{\alpha_{\mathfrak J} }\circ i&=i\circ \beta \label{q1}\\
			\alpha\circ \pi&=\pi \circ \overline{\alpha_{\mathfrak J} }\label{q2}\\
			i(\mathfrak{a})&=\ker\pi \label{ker}\\
			B(i(v),i(w))&=B(v,w)\label{q3}
		\end{align}

		$\overline{\alpha_{\mathfrak J} }\circ i=i\circ \beta  $, $ \alpha\circ \pi=\pi \circ \overline{\alpha_{\mathfrak J} } $, 
		$ i(\mathfrak{a})=\ker\pi$

		We assume that 	$ \left(\mathfrak J,I,i,\pi \right)  $  is a quadratic extension.
		Then, by \eqref{quadratic} $ \ker(\pi)=Im(i)$, $\alpha\circ \pi=\pi \circ \overline{\alpha_{\mathfrak J} }$. Moreover, we have $Im(i)=I^{\perp}/I $.\\
		First, we show that  		$\alpha_{\mathfrak J}^* \circ \tilde{\pi}^*=\tilde{\pi}^*\circ \alpha^*  $.	
		We have, \[ \alpha\circ \pi=\pi \circ \overline{\alpha_{\mathfrak J} }=\pi \circ \pi_{n}\circ \alpha_{\mathfrak J}=\tilde{\pi}\circ\alpha_{\mathfrak J}. \]					 
		Hence $( \alpha\circ \pi)^*=(\tilde{\pi}\circ\alpha_{\mathfrak J})^*$. Then 					 
		$\pi^*\circ \alpha^*=\alpha_{\mathfrak J}^*\circ \tilde{\pi}^*.$\\
		By definition of $\overline{\alpha_{\mathfrak J}}$, we have $\overline{\alpha_{\mathfrak J}} \circ \pi_{n}=\pi_{n} \circ \alpha_{\mathfrak J} $.\\
		Now, we show that $Im(\tilde{\pi}^*)=\ker(\pi_n)$. By $\ker\pi= i(\mathfrak{a})=I^{\perp}/I $ and $  \tilde{\pi}=\pi\circ \pi_n $ we obtain 
		$ \ker(\tilde{\pi})=I^{\perp} $.
		Since $ Im(\tilde{\pi}^*)=\left( \ker(\tilde{\pi})\right)^{\perp}  $, one can deduce $Im(\tilde{\pi}^*)=I$. So $Im(\tilde{\pi}^*)=\ker(\pi_n)$
		and the sequence 
		$$0\longrightarrow J^*\stackrel{\tilde{\pi}^*}{\longrightarrow} \mathfrak J^*\cong \mathfrak J\stackrel{\pi_n}{\longrightarrow } \mathfrak J/I \longrightarrow 0,$$	
		define an extension of $ \mathfrak J/I $ by $ J^* $.\\			
		Conversely, we assume that the following sequence
		$$0\longrightarrow J^*\stackrel{\tilde{\pi}^*}{\longrightarrow}\mathfrak J^*\cong \mathfrak J\stackrel{\pi_n}{\longrightarrow } \mathfrak J/I \longrightarrow 0,$$	
		define an extension.Then, 
		$\alpha_{\mathfrak J}^*\circ \tilde{\pi}^*=\pi^*\circ \alpha^*$,
		$\overline{\alpha_{\mathfrak J}}\circ \pi_{n}=\pi_{n}\circ \alpha_{\mathfrak J}$		and 	$Im(\tilde{\pi}^*)=\ker(\pi_n).$
		We have $Im(\tilde{\pi}^*)=\left( \ker(\tilde{\pi})\right)^{\perp}    $,  
		$Im(\tilde{\pi}^*)=\ker(\pi_n)$	and	$\ker(\pi_n)=I$. Hence, $\ker(\tilde{\pi})=I^\perp$	and $ I\subset  I^{\perp} $. Then 
		$\ker(\pi)=I^{\perp}/I$. By \eqref{ker}, we have $ Im(i)=\ker(\pi)=I^{\perp}/I$.
		Moreover, we have \eqref{q1}, \eqref{q2} and \eqref{q3}. Therefore, 
		$ \left(\mathfrak J,I,i,\pi \right)  $  is a quadratic extension.                  
	\end{proof}			
	\begin{defn}
		Two quadratic extensions $ \left(\mathfrak J_1,I_1,i_1,\pi_1\right)  $, $ \left( \mathfrak J_2,I_2,i_2,\pi_2\right)  $ of $ J $ by $ \mathfrak a $ are called
		to be equivalent if there exists an isomorphism of metric Lie algebras
		$ \Phi\colon \mathfrak J_1\to \mathfrak J_2 $ 
		which
		maps $ i_1  $ onto $ i_2 $ and satisfies
		$ \overline{\Phi}\circ i_1=i_2 $
		and
		$\pi_2 \circ\overline{\Phi} =\pi_1 $,
		where $\overline{\Phi}\colon \mathfrak J_1/I_1\to  \mathfrak J_2/I_2 $
		is the induced map.                                                    
	\end{defn}			
	\subsection{Twofold extensions}	
	The twofold extension of Lie  algebras was studied  in \cite{I.KATH2} (also called Standard model in \cite{I.KATH}). In the following, we define and study the Twofold extension of Hom-Jacobi-Jordan  algebras.
	
	Let $ \left(J,[\cdot,\cdot],\alpha \right)  $ a Hom-Jacobi-Jordan  algebra  and let $ \left(\mathfrak a,\rho,\beta,B_{\mathfrak a} \right)  $ 
	be a quadratic representation of $J$. For each $(\theta,\gamma)\in Z^2_{Q(J,\mathfrak a)}$ we want to define  structures of a metric Hom-Jacobi-Jordan algebra on the vector space $\mathfrak d_{\theta,\gamma}:=J\oplus  \mathfrak a\oplus J^* $. Let $ \alpha_{\mathfrak d_{\theta,\gamma}}=\alpha+\beta+\alpha^* $. Define a bracket on $\mathfrak d_{\theta,\gamma}$ by
	\begin{equation*}
		\begin{aligned}
			[x,y]_{\theta,\gamma}&=\delta(x,y)+\theta(x,y)+\gamma(x,y,\cdot);\\
			[x,v]_{\theta,\gamma}&=\rho(x)v +B_{\mathfrak a}\left(\theta(\cdot,x),v \right) ;\\
			[v,w]_{\theta,\gamma}&=B_{\mathfrak a}\left(\rho(\cdot)v,w \right);\\
			[Z,x]_{\theta,\gamma}&=Z\left(\delta(x,\cdot) \right);\\
			[Z_1,v+Z_2]_{\theta,\gamma}&=0.      
		\end{aligned}
	\end{equation*}
	Define a  symmetric bilinear form $B $ on $ d_{\theta,\gamma} $ by
	\begin{equation*}\label{BJ*}
		B(x+v+Z_{1},y+w+Z_{2})=Z_{1}(y)+Z_{2}(x)+B_{\mathfrak a}\left(v,w \right)
	\end{equation*} 
	for all  $x,y\in J,\, v,w\in \mathfrak{a},\ Z_{1},Z_{2}\in J^*$.
	Define the linear map $i_{0}\colon \mathfrak{a}_{\theta,\gamma}\to \mathfrak{d}_{\theta,\gamma}/J^*$ by $ i_{0}(v)=v+J^* $ and the linear map $\pi_{0} \colon \mathfrak{d}_{\theta,\gamma}/J^*\to J$ by $\pi_{0}(x+v+J^*)=x$.
	\begin{prop}\label{theoremTsquadratic}
		With the above notations, the quadruple $ \left(\mathfrak{d}_{\theta,\gamma},J^*,i_{0},\pi_{0} \right) $
		define a quadratic extension.	            
	\end{prop}
	\begin{proof}
		We only prove that 
		$\left(\mathfrak{d}_{\theta,\gamma},	[\cdot,\cdot]_{\theta,\gamma},\alpha_{\mathfrak d_{\theta,\gamma}},B \right) $ is a metric Hom-Jordan-Jacobi algebra.
		Denote $ \mathfrak{d}_{\theta,\gamma}=\mathfrak n $ and define a trilinear form $\gamma_{\mathfrak n}$ on $\mathfrak n$ by $\gamma_{\mathfrak n}\left(a,b,c \right)=B\left(d'(a,b),c \right)  $ for all $a,b,c\in \mathfrak n$.  Using Theorem\ref{gamma},  it  is sufficient to  show that $\gamma_{\mathfrak n}$ is symmetric and $d_{r}^3\gamma_{\mathfrak n}=0$.\\
		We have \begin{gather*}
			\gamma_{\mathfrak n}(x,y,z)=B\left([x,y]_{\theta,\gamma},z \right)
			=B\left([x,y]+\theta(x,y)+\gamma(x,y,\cdot),z \right)
			=\gamma(x,y,z).
		\end{gather*}
		Since $\gamma$ is symmetric, we obtain that the restriction of $\gamma_{\mathfrak n}$ to $J^3$ is symmetric.\\
		For all $ x,y\in J $, $ v\in \mathfrak a $, we have 
		\begin{align*}
			\gamma_{\mathfrak n}(x,y,v)&=B\left([x,y]_{\theta,\gamma},v \right)
			=B_{\mathfrak a}(\theta(x,y),v);\\
			\gamma_{\mathfrak n}(x,v,y)&=B\left([x,v]_{\theta,\gamma},y \right)
			=B_{\mathfrak a}(\theta(x,y),v).
		\end{align*}	
		Therefore, using the fact that $[x,y]_{\theta,\gamma}=[y,x]_{\theta,\gamma}  $ and $ [x,v]_{\theta,\gamma} =[v,x]_{\theta,\gamma}$, one can deduce that 
		the restriction of $\gamma_{\mathfrak n}$ to $J^2\times V$ is symmetric.\\
		For all $ x\in J $, $ v,w\in \mathfrak a $, we have 
		\begin{align*}
			\gamma_{\mathfrak n}(x,v,w)&=B\left([x,v]_{\theta,\gamma},w \right)
			=B_{\mathfrak a}(\rho(x)w,v)\\
			\gamma_{\mathfrak n}(v,w,x)&=B\left([v,w]_{\theta,\gamma},x \right)
			=B_{\mathfrak a}\left(\rho(x)v,w \right)
		\end{align*}
		and since $ \left(\mathfrak a,\rho,\beta,B_{\mathfrak a} \right)  $ is a quadratic representation of $J$ on $\mathfrak a$, then
		the restriction of $\gamma_{\mathfrak n}$ to $J\times V^2$ is symmetric.\\
		For all $ u, v,w\in \mathfrak a $, we have 
		\begin{gather*}
			\gamma_{\mathfrak n}(v,w,u)=B\left([v,w]_{\theta,\gamma},u \right)
			=B\left(B_{\mathfrak a}\left(\rho(\cdot)v,w \right),u \right)
			=0.
		\end{gather*}
		Thus,  the restriction of $\gamma_{\mathfrak n}$ to $ V^3$ is symmetric too.\\
		For all $x,y,z,a\in J$, we have
		\begin{align}
			&d^3_{r}\gamma_{\mathfrak n}(x,y,z,t)\nonumber\\
			&	=\gamma\left( \delta(x,y),\alpha(z),t \right)
			+\gamma\left( \delta(x,z),\alpha(y),t \right)
			+\gamma\left( \delta(y,z),\alpha(x),t \right)\label{Sep22Gamma1}\\
			+&\gamma(x,y,\delta(\alpha(z),t )	
			+\gamma(x,z,\delta(\alpha(y),t) )	
			+\gamma(y,z,\delta(\alpha(x),t) )\label{Sep22Gamma2}	\\
			+&\gamma(\alpha(z),t,\delta(x,y) )	
			+\gamma(\alpha(y),t,\delta(x,z) )	
			+\gamma(\alpha(x),t,\delta(y,z) )\label{Sep22Gamma3}	\\
			+&\gamma\left( \delta(\alpha(z),t),x,y \right)
			+\gamma\left( \delta(\alpha(y),t),x,z \right)
			+\gamma\left( \delta(\alpha(x),t),y,z \right)\label{Sep22Gamma4}\\
			+& B_{\mathfrak a}\left(\theta (y,x),\theta(\alpha(z),t) \right)
			+ B_{\mathfrak a}\left(\theta (z,x),\theta(\alpha(y),t) \right)
			+ B_{\mathfrak a}\left(\theta (z,y),\theta(\alpha(x),t) \right)\label{Sep22Gamma5}\\
			+& B_{\mathfrak a}\left(\theta (t,\alpha(z)),\theta(x,y) \right)
			+ B_{\mathfrak a}\left(\theta (t,\alpha(y)),\theta(x,z) \right)
			+ B_{\mathfrak a}\left(\theta (t,\alpha(x)),\theta(y,z) \right)\label{Sep22Gamma6}
		\end{align}	
		where $ t=\alpha(a) $.  Since $\gamma$ is symmetric, we  get \[\eqref{Sep22Gamma1}+\eqref{Sep22Gamma2}= d_{r}\gamma(x,y,z,t)  \text{ and  } 
		\eqref{Sep22Gamma3}+\eqref{Sep22Gamma4}=d_{r}\gamma(x,y,z,t). \]
		Since $\theta  $ is a $2$-Hom-cochain, by Proposition \ref{1LemmaSep17}, we obtain
		\[   \eqref{Sep22Gamma5}+ \eqref{Sep22Gamma6}=B_{\mathfrak a}\left(\theta\wedge ( \theta\circ \alpha)\right)(x,y,z,a).\]
		Thus, $d^3_{r}\gamma_{\mathfrak n}(x,y,z,t)=2d_{r}\gamma(x,y,z,t)+B_{\mathfrak a}\left(\theta\wedge ( \theta\circ \alpha)\right)(x,y,z,a)  $.\\
		Then, since $(\theta,\gamma) $ is a  quadratic $2$-cocycle, we obtain $ d^3_{r}\gamma_{\mathfrak n}(x,y,z,t)=0. $\\
		By straightforward computations, for all $ x,y,z\in J $, $ v\in \mathfrak a$, we have
		\begin{align*}
			&\frac{1}{2}d^3\gamma_{\mathfrak n}(x,y,z,v)\\
			&=B_{\mathfrak a}\left( \theta([x,y],\alpha(z)),v\right)
			+B_{\mathfrak a}\left( \theta([x,z],\alpha(y)),v\right)
			+ B_{\mathfrak a}\left(\theta([y,z],\alpha(x)),v \right) \\
			+&B\left( \rho(\alpha(z))\theta(x,y),v\right)
			+B\left( \rho(\alpha(x))\theta(y,z),v\right)
			+B\left( \rho(\alpha(y))\theta(x,z),v\right)\\
			&=\frac{1}{2}B_{\mathfrak a}\left(d^2\theta(x,y,z),v\right).	
		\end{align*}
		Therefore, $d^3\gamma_{\mathfrak n}(x,y,z,v)=0  $ by $(\theta,\gamma) $ is a  quadratic $2$-cocycle.\\
		Similarly, for any $ x,y\in J $, $u, v\in \mathfrak a$, we get
		\begin{align}
			&d^3\gamma_{\mathfrak n}(x,y,u,v)\nonumber\\
			&=B_{\mathfrak n}\Big(u
			,\beta(\rho([x,y])v)  \Big)
			+B_{\mathfrak a}\left(\rho(\alpha(y))\rho(x)u,v \right) 
			+B_{\mathfrak a}\left(\rho(\alpha(x))\rho(y)u,v \right)\label{aproof}\\
			+&B\Big(\rho(y)\rho(\alpha(x))v  ,u     \Big)
			+B\Big(\rho(x)\rho(\alpha(y))v  ,u     \Big)
			+B_{\mathfrak a}\left(u,\left( \beta\rho( [x, y] )v\right)  \right).\label{bproof}	
		\end{align}
		Therefore, by \eqref{coadjointprime1} (resp.\eqref{coadjointprime2} ), we have 
		$\eqref{aproof}=0$ (resp $\eqref{bproof}=0$ ).
		Hence,				
		$ d^3\gamma_{\mathfrak n}(x,y,u,v)=0 $.\\
		For all $ x\in J $, $u, v,w, s\in \mathfrak a$, $Z\in J^*$, by $ B\left(Z,u \right)=0  $, we have $  d^3\gamma_{\mathfrak n}(u,v,w,x) =0  $,
		$  d^3\gamma_{\mathfrak n}(u,v,x,w) =0  $ and $  d^3\gamma_{\mathfrak n}(u,v,w,s) =0  $. The rest of the proof is straightforward.	              
	\end{proof}										
	\begin{defn}				
		We denote the quadratic extension$ \left(\mathfrak{d}_{\theta,\gamma},J^*,i_{0},\pi_{0} \right) $
		constructed in The Proposition  \ref{theoremTsquadratic}
		by $ \mathfrak d_{\theta,\gamma}(\mathfrak{a},J,\rho) $ and call it a twofold extension.
	\end{defn}
	\subsection{Classification by cohomology}
	\begin{defn}
		Two quadratic extensions $ \left(\mathfrak J_1,I_1,i_1,\pi_1\right)  $, $ \left( \mathfrak J_2,I_2,i_2,\pi_2\right)  $ of $ J $ by $ \mathfrak a $ are called
		to be equivalent if there exists an isomorphism of metric Hom-Jacobi-Jordan  algebras
		$ \Phi\colon \mathfrak J_1\to \mathfrak J_2 $ 
		which
		maps $ i_1  $ onto $ i_2 $ and satisfies
		$ \overline{\Phi}\circ i_1=i_2 $
		and
		$\pi_2 \circ\overline{\Phi} =\pi_1 $,
		where $\overline{\Phi}\colon \mathfrak J_1/I_1\to  \mathfrak J_2/I_2 $
		is the induced map.                                                    
	\end{defn}
	\begin{prop}\label{bijectionApril}
		Any quadratic extension $ \left(\mathfrak J,I,i,\pi \right)  $ is equivalent to a  Twofold extension $\left( \mathfrak{d}_{\theta,\gamma},J^*,i_{0},j_{0} \right)  $ 
	\end{prop}
	\begin{proof}
		Let $$\mathcal E:0\longrightarrow \mathfrak{a}\stackrel{i}{\longrightarrow} \mathfrak J/I\stackrel{\pi}{\longrightarrow } J \longrightarrow 0,$$
		be the 
		extension of $ J $ defined in \eqref{qEXTENSION} and $s$ a section of $ \mathcal E $. 	Then we have by Proposition \ref{saadaoui}, $\mathfrak J/I=s(J)\oplus i(\mathfrak a)$ and the extension $ \mathfrak e $ is equivalent to  
		$$0\longrightarrow \left(\mathfrak{a} ,\rho,\beta\right)  \stackrel{i_{0}}{\longrightarrow}\left(  M,[\cdot,\cdot]_{\theta},\alpha_{M}\right)  \stackrel{\pi_{0}}{\longrightarrow } \left( J,[\cdot,\cdot],\alpha\right)  \longrightarrow 0$$ 
		where $ \theta  $ is a $2$-cocyle of $J$ on $\mathfrak{a}$ and  $M=J\oplus \mathfrak{a}.   $\\
		Now, we consider the extension
		$$\mathcal E^*:0\longrightarrow J^*\stackrel{\tilde{\pi}^*}{\longrightarrow} \mathfrak J\stackrel{\pi_n}{\longrightarrow } \mathfrak J/I \longrightarrow 0,$$			
		be the  extension   defined in \eqref{qEXTENSION2} and $ s' $ a section of $ \mathcal E^* $.
		Then we have by Proposition \ref{saadaoui}, 
		$\mathfrak J=s'(J/I )\oplus  \tilde{\pi}^*(J^*)  $ and 	
		the extension $ \mathcal E $ is equivalent to  
		$$0\longrightarrow \left(J^* ,\rho',\alpha'\right)  \stackrel{i_{0}}{\longrightarrow}\left(  M',[\cdot,\cdot]_{\gamma'},\alpha_{M'}\right)  \stackrel{\pi_{0}}{\longrightarrow } \left( \mathfrak J/I ,[\cdot,\cdot]_{\overline{\mathfrak J }},\overline{\alpha_{\mathfrak J} }\right)  \longrightarrow 0$$ 
		where $ \gamma' $ is a $2$-cocycle of $ \mathfrak J/I $ on $J^*$
		and
		$M'=\mathfrak J/I\oplus J^*.   $ \\
		We have $\mathfrak J=s'(J/I )\oplus  \tilde{\pi}^*(J^*) =s'(s(J)\oplus i(\mathfrak a))\oplus  \tilde{\pi}^*(J^*) $.
		We can write $ \pi_n\colon s'(\mathfrak J/I)\to \mathfrak J/I $ and $ \pi\colon s(J)\to J $. Hence 
		$\tilde{\pi}^*(J^*)=\left( s's(J) \right) ^*$                      
		
		Using $\mathfrak J=s'(J/I )\oplus  \tilde{\pi}^*(J^*)  $ and $ \tilde{\pi}^*(J^*)=\left( s's(J) \right) ^*$, we obtain $\mathfrak J= s's(J) \oplus s' i(\mathfrak{a})  \oplus  \left( s's(J)\right)^* $.\\	
		Then, using \textbf{Proposition} \ref{1Sep23}, for  all $x\in  J$, $v\in  \mathfrak a$,  $Z\in   \mathfrak J^*$ we have   
		\begin{align*}
			[s's(x),s's(y)]_{\mathfrak J}&=[s's(x),s's(y)]_{s's(J)}+\theta(s's(x),s's(y))+\gamma(s's(x),s's(y),\cdot);\\
			[s's(x),s'i(v)]_{\mathfrak J}&=\rho(s's(x))v+B_\rho(s'i(v),\alpha(s's(x),\cdot));\\
			[s'i(v),s'i(w)]_{\mathfrak J}&=B_{\mathfrak{a}}(\rho(\cdot)(s'i(v)),s'i(w));\\
			[Z,s's(x)]_{\mathfrak J}&=Z(\delta(s's(x),\cdot));\\
			[Z_1,s'i(v)+Z_2]_{\mathfrak J}&=0.
		\end{align*}
		
		Now, define the linear map $ \Psi\colon  J \oplus \mathfrak{a}  \oplus  J^*\to \mathfrak J$ by $\Psi(x+v+Z)=s's(x)+s'i(v)+(s's)^{*}(Z) $ and the bilinear 
		map $ [\cdot,\cdot]_{\mathfrak d}\colon J \oplus \mathfrak{a}  \oplus  J^*\to J \oplus \mathfrak{a}  \oplus  J^* $ by \[[x+v+Z,y+w+Z']_{\mathfrak d} =\Psi^{-1}\left( [s's(x)+s'i(v)+(s's)^{*}(Z),s's(y)+s'i(w)+(s's)^{*}(Z')]_{\mathfrak J}\right)  .	\]
		Then
		\begin{align*}\label{ext2021}
			&\left[ \Psi(x+v+Z),\Psi(y+w+Z')\right]_{\mathfrak J}\\
			&= \left[ s's(x)+s'i(v)+(s's)^{*}(Z),s's(y)+s'i(w)+(s's)^{*}(Z')\right]_{\mathfrak J}\\
			&=\Psi\left([x+v+Z,y+w+Z']_{\mathfrak d} \right).          
		\end{align*} 
		Moreover, we have $ \overline{\Psi}\circ i_0(v)=i(v) $  and $ \pi\circ \overline{\Psi}(\overline{x}) = \pi\circ s(x)=x=\pi_0(x)$. 
	\end{proof}
\begin{lem}\label{H2}
	Let  $\mathfrak d_{\theta,\gamma}:= \mathfrak d_{\theta,\gamma}(\mathfrak{a},J,\rho) $ and   $\mathfrak d_{\theta',\gamma'}:= \mathfrak d_{\theta',\gamma'}(\mathfrak{a},J,\rho) $ two Twofold  extensions such that $ d^2_{Q}(\theta,\gamma) =d^2_{Q}(\theta',\gamma')$.		
	Then, the Twofold extensions 
	$\mathfrak d_{\theta,\gamma}:= \mathfrak d_{\theta,\gamma}(\mathfrak{a},J,\rho) $ and   $\mathfrak d_{\theta',\gamma'}:= \mathfrak d_{\theta',\gamma'}(\mathfrak{a},J,\rho) $		
	are equivalents.             
\end{lem}
\begin{proof}
	Define the linear map $\Phi\colon J\oplus \mathfrak a\oplus J^*\to J\oplus \mathfrak a\oplus J^* $  by \[\Phi(x+v+Z)=x+\underbrace{ v -\tau(x)}_{\in \mathfrak a}+\underbrace{  Z
		-\frac{1}{2}B_{\mathfrak a}\big(\tau(x),\tau(\cdot)\big)
		+B_{\mathfrak a}(v,\tau(\cdot))}_{\in J^*}.\] \\
	We have	
	\begin{align*}
		\Phi(\alpha(x)+\beta(v)+\alpha'(Z))&=\alpha(x)+ \beta(v) - \tau(\alpha(x))  +\alpha' (Z)
		-\frac{1}{2}B_{\mathfrak a}\big(\tau(\alpha(x)),\tau(\cdot)\big)
		+B_{\mathfrak a}(\beta(v),\tau(\cdot)) \\
		&=\alpha(x)+ \beta(v) - \beta(\tau(x))  +\alpha' (Z)
		-\frac{1}{2}B_{\mathfrak a}\big(\beta(\tau(x)),\tau(\cdot)\big)
		+B_{\mathfrak a}(\beta(v),\tau(\cdot))\\
		&=\alpha(x)+ \beta(v) - \beta(\tau(x))  +\alpha' (Z)
		-\frac{1}{2}B_{\mathfrak a}\big(\tau(x),\beta(\tau(\cdot))\big)
		+B_{\mathfrak a}(v,\beta(\tau(\cdot)))\\
		&=\alpha(x)+ \beta(v) - \beta(\tau(x))  +\alpha' (Z)
		-\frac{1}{2}B_{\mathfrak a}\big(\tau(x),\tau(\alpha(\cdot))\big)
		+B_{\mathfrak a}(v,\tau(\alpha(\cdot)))\\
		&=\alpha(x)+ \beta\left( v- \tau(x)\right)   +\alpha' \left( Z
		-\frac{1}{2}B_{\mathfrak a}\big(\tau(x),\tau(\cdot)\big)
		+B_{\mathfrak a}(v,\tau(\alpha(\cdot)))\right) .	            
	\end{align*}	            
	Hence, $\Phi\circ (\alpha+\beta+\alpha')=(\alpha+\beta+\alpha')\circ \Phi $.

	We have
	\begin{equation*}
		\begin{aligned}
			[x,y]_{\theta,\gamma}&=[x,y]+\theta(x,y)+\gamma(x,y,\cdot);\\
			[x,v]_{\theta,\gamma}&=\rho(x)v +B_{\mathfrak a}\left(\theta(\cdot,x),v \right) ;\\
			[v,w]_{\theta,\gamma}&=B_{\mathfrak a}\left(\rho(\cdot)v,w \right);\\
			[Z,x]_{\theta,\gamma}&=Z\left([x,\cdot] \right);\\
			[Z_1,v+Z_2]_{\theta,\gamma}&=0.      
		\end{aligned}
	\end{equation*}
	Using 
	Proposition \ref{17octequiv},
	we have 	$	\theta'=\theta+d_{r}^1\tau $ and $
	\gamma'=	\gamma+d_{r}^2\sigma  -	B\left( \tau\wedge (\theta+\frac{1}{2}d^1\tau )\right) $. Hence, the structure $[\cdot,\cdot]_{\theta',\gamma'}$ of the Twofold  extension $\mathfrak d_{\theta',\gamma'}:= \mathfrak (\mathfrak{a},J,\rho) $ is given by
	\begin{align*}
		[x,y]_{\theta',\gamma'}&=[x,y]+\theta(x,y)+d^1\tau(x,y)+\gamma(x,y,\cdot)\\
		+&d^2\sigma  (x,y,\cdot )-	B((\theta+\frac{1}{2}d^1\tau)\wedge \tau  )(x,y,\cdot )
		;\\
		[x,v]_{\theta',\gamma'}&=\rho(x)v +B_{\mathfrak a}\left(\theta(\cdot,x)+d^1\tau(\cdot,x),v \right) ;\\
		[v,w]_{\theta',\gamma'}&=B_{\mathfrak a}\left(\rho(\cdot)v,w \right);\\
		[Z,x]_{\theta',\gamma'}&=Z\left([x,\cdot] \right);\\
		[Z_1,v+Z_2]_{\theta,\gamma}&=0.
	\end{align*}	
	
	We have
	\begin{align*}
		\Phi\left( [x,y]_{\theta',\gamma'}\right)
		&=[x,y]+\theta(x,y)+d^1\tau (x,y)
		+\gamma(x,y,\cdot )\\
		+&d^2\sigma(x,y,\cdot) -B_{\mathfrak a}\big((\theta +\frac{1}{2}d^1\tau)\wedge \tau \big)(x,y,\cdot)\\
		-&\tau\left([x,y] \right)-\sigma\left( [x,y],\cdot\right)
		+\frac{1}{2}B_{\mathfrak a}\big(\tau([x,y],\tau(\cdot)\big)\\
		+&B_{\mathfrak a}\left( (\theta(x,y)+d^1\tau(x,y),\tau(\cdot)\right)  
	\end{align*}
	Hence, by \eqref{operatorCoadjoint4} , \eqref{Wedge} and \eqref{EqQrepr1}, we obtain
	\begin{align*}
		\Phi\left( [x,y]_{\theta',\gamma'}\right)&	=[x,y]+\theta(x,y)+\gamma(x,y,\cdot )
		-\rho(x)\tau(y)-\rho(y)\tau(x) \\
		-&\sigma\left(y,[x,\cdot]  \right) 
		-\sigma\left(x,[y,\cdot] \right)\\ 
		-&B_{\mathfrak a}\left(\theta(x,\cdot) , \tau(y) \right)
		-B_{\mathfrak a}\left(\theta(y,\cdot) , \tau(x) \right)\\	
		-&\frac{1}{2}B_{\mathfrak a}\Big(\tau\left( [x,\cdot]\right) ,\tau(y) \Big)
		-\frac{1}{2}B_{\mathfrak a}\Big(\tau\left( [y,\cdot]\right) ,\tau(x) \Big)	
		+B_{\mathfrak a}\Big(\rho(\cdot)\tau(x) ,\tau(y) \Big).			
	\end{align*}
	In other hand, we have
	\begin{align*}
		&	[\Phi(x),\Phi(y)]_{\theta,\gamma}  \\
		&=\Big[x-\tau(x)-\sigma(x,\cdot)
		-\frac{1}{2}B_{\mathfrak a}\big(\tau(x),\tau(\cdot)\big)
		, y-\tau(y)-\sigma(y,\cdot)
		-\frac{1}{2}B_{\mathfrak a}\big(\tau(y),\tau(\cdot)\big)            \Big]_{\theta,\gamma}\\
		&=[x,y]+\theta(x,y)+\gamma(x,y,\cdot)
		-\rho(x)\tau(y) -B_{\mathfrak a}\left(\theta(\cdot,x),\tau(y) \right)\\	
		-&\sigma(y, [x,\cdot] )
		-\frac{1}{2}B_{\mathfrak a}\Big(\tau(y),\tau([x,\cdot ]   )\Big)
		-\rho(y)\tau(x) -B_{\mathfrak a}\left(\theta(\cdot,y),\tau(x) \right)\\	
		+&B_{\mathfrak a}\left(\rho(\cdot)\tau(x),\tau(y) \right)
		-\sigma(x,[y,\cdot])
		- \frac{1}{2}B_{\mathfrak a}\Big(\tau(x),\tau\left( [y, \cdot]\right)  \Big). 
	\end{align*}
	Therefore, $	\Phi\left( [x,y]_{\theta',\gamma'}\right)= [\Phi(x),\Phi(y)]_{\theta,\gamma} . $\\
	Similar, we show that $\Phi\left( [x,w]_{\theta',\gamma'}\right)
	=[\Phi(x),\Phi(w)]_{\theta,\gamma}  $, $ \Phi\left( [x,Z]_{\theta',\gamma'}\right)=
	[\Phi(x),\Phi(Z)]_{\theta,\gamma} $, $ \Phi\left( [v,w]_{\theta',\gamma'}\right)
	=[\Phi(v),\Phi(w)]_{\theta,\gamma} $.\\
	Finally, we show that $\mathfrak d_{\theta,\gamma}$	 and $\mathfrak d_{\theta',\gamma'}$ are $ i $-isomorphic.				
	\begin{align*}
		B(\Phi(x),\Phi(y))
		&= B\left(x-\tau(x)
		-\frac{1}{2}B_{\mathfrak a}\big(\tau(x),\tau(\cdot)\big),
		y-\tau(y)
		-\frac{1}{2}B_{\mathfrak a}\big(\tau(y),\tau(\cdot)\big)\right)\\
		&=B_{\mathfrak a}\left(\tau(x),
		\tau(y)\right)
		-\frac{1}{2}B_{\mathfrak a}\big(\tau(y),\tau(x)\big)	
		-\frac{1}{2}B_{\mathfrak a}\big(\tau(x),\tau(y)\big)\\
		&=0=B(x,y)
	\end{align*}	
	\begin{gather*}
		B(\Phi(x),\Phi(v))=B(x-\tau(x)
		-\frac{1}{2}B_{\mathfrak a}\big(\tau(x),\tau(\cdot)\big),v+B_{\mathfrak a}(v,\tau(\cdot)))\\
		=-B_{\mathfrak a}(\tau(x),v)+B_{\mathfrak a}(v,\tau(x))=0
	\end{gather*}	
	\begin{align*}
		B(\Phi(u),\Phi(v))&=B(u+B_{\mathfrak a}(u,\tau(\cdot))
		,v+B_{\mathfrak a}(v,\tau(\cdot)))\\
		&=B_{\mathfrak a}(u,v)
	\end{align*}	
\end{proof}
\begin{lem}
	Let $\mathfrak d_{\alpha,\gamma}:= \mathfrak d_{\theta,\gamma}(\mathfrak{a},J,\rho) $ and   $\mathfrak d_{\theta',\gamma'}:= \mathfrak d_{\theta',\gamma'}(\mathfrak{a},J,\rho) $ be two equivalent Twofold extensions. Then, the quadratic $2$-cocycle $ \left(\theta-\theta',\gamma-\gamma' \right)  $ is trivial.
\end{lem}
\begin{proof}

	Let 	$\Phi(x)=f(x)+\tau(x)+\zeta(x) $ where $f\colon J\to J$, $\tau \colon J\to \mathfrak a$ and $\zeta \colon J\to J^*.  $ 	
	Using $ \pi\ \circ\ \Phi'= \pi $, we obtain $ f(x)=x $. Then
	\[ \Phi(x)=x+\tau(x)+\zeta(x) . \]
	Let $\Phi(v)=g(v)+h(v)+\eta(v)$, 	where $g\colon\mathfrak a \to J$, $h \colon \mathfrak a\to \mathfrak a$ and $\eta \colon\mathfrak a \to J^*.$ Using $ \Phi' \circ \ i= i $, we obtain $ g(v)=0 $  and $ h(v)=v $, Then
	$\Phi(v)=v+\eta(v)$. 
	Using $ B(v,x)=B(\Phi(v),\Phi(x))  $, we obtain $ \eta(v)(x) =-B_{\mathfrak a}( v  ,\tau(x))  $.
	With $\Phi$ is an isometry and $\Phi(J^*) \subset J^* $ we obtain
	$ \Phi(Z)=Z. $\\
	Using $B(\Phi(x),\Phi(y)) = B(x,y) $ we obtain 
	$	B_{\mathfrak a}(\tau(x),\tau(y))=- \zeta(x)(y)-\zeta(y)(x)$. Since $\zeta(x)(y)=\zeta(y)(x)$, we obtain $ \zeta(x,y)=-\frac{1}{2}	B_{\mathfrak a}(\tau(x),\tau(y)). $\\
	By $ \Phi(d(x,y))=d'(\Phi(x),\Phi(y) ) $, we obtain 
	\begin{gather*}
		\theta(x,y)=\theta'(x,y)-\tau(\delta(x,y))+\rho(x)\tau(y) +\rho(y)\tau(x)=\theta'(x,y)-d^1\tau(x,y)
	\end{gather*}
	and
	\begin{gather*}
		\gamma(x,y,\cdot)
		=\gamma'(x,y,\cdot)-B_{\mathfrak a}\left(   ( \theta'+\frac{1}{2}d(-\tau))  \wedge (-\tau)\right)(x,y,\cdot).
	\end{gather*}
	
	Hence,
	\begin{gather*}
		\left\{ \begin{array}{ll}
			\alpha=\alpha'+d^1\left( -\tau\right)  \\
			\gamma=\gamma'(x,y,\cdot)-B_{\mathfrak a}\left(   ( \alpha'+\frac{1}{2}d(-\tau))  \wedge (-\tau)\right)(x,y,\cdot)
		\end{array} \right.
	\end{gather*} 
	Using Proposition \ref{structureMAY24}, we have $d_{c}^2\theta=0$.\\
	Therefore, using 
	Proposition \ref{17octequiv}, we have
	$
	d^2_{Q}	(\theta,\gamma)=d^2_{Q}	(\theta',\gamma') $.	
\end{proof}
\begin{thm}\label{classiJune4}
	The set $ Ext( J,\mathfrak a) $  of equivalence classes of  Twofold extensions $ \left( J,\delta,\alpha \right)  $  by an abelian   $\mathfrak a$
	is one-to-one   correspondence with $ Z_{Q}^2( J,\mathfrak a)/ B^2(J,\mathfrak a)$, that is 
	\[Ext( J,\mathfrak a) \cong  H_{Q}^2( J,\mathfrak a). \]	                     
\end{thm}

					
\end{document}